\documentclass[11pt,oneside,english]{amsart}
\usepackage[T1]{fontenc}
\usepackage[latin9]{inputenc}
\usepackage[a4paper]{geometry}
\usepackage{babel}
\usepackage{amstext}
\usepackage{amsthm}
\usepackage{amssymb}
\usepackage{setspace}
\usepackage[unicode=true,pdfusetitle,
 bookmarks=true,bookmarksnumbered=false,bookmarksopen=false,
 breaklinks=false,pdfborder={0 0 1},backref=false,colorlinks=false]
 {hyperref}

\makeatletter
\numberwithin{equation}{section}
\numberwithin{figure}{section}
\theoremstyle{plain}
\newtheorem{thm}{\protect\theoremname}[section]
\theoremstyle{remark}
\newtheorem{rem}[thm]{\protect\remarkname}
\theoremstyle{plain}
\newtheorem{conjecture}[thm]{\protect\conjecturename}
\theoremstyle{plain}
\newtheorem{question}[thm]{\protect\questionname}
\theoremstyle{definition}
\newtheorem{defn}[thm]{\protect\definitionname}
\theoremstyle{remark}
\newtheorem*{acknowledgement*}{\protect\acknowledgementname}
\theoremstyle{plain}
\newtheorem{lem}[thm]{\protect\lemmaname}
\theoremstyle{plain}
\newtheorem{prop}[thm]{\protect\propositionname}
\theoremstyle{plain}
\newtheorem{cor}[thm]{\protect\corollaryname}
\theoremstyle{definition}
\newtheorem{example}[thm]{\protect\examplename}

\makeatother

\providecommand{\acknowledgementname}{Acknowledgement}
\providecommand{\conjecturename}{Conjecture}
\providecommand{\corollaryname}{Corollary}
\providecommand{\definitionname}{Definition}
\providecommand{\examplename}{Example}
\providecommand{\lemmaname}{Lemma}
\providecommand{\propositionname}{Proposition}
\providecommand{\questionname}{Question}
\providecommand{\remarkname}{Remark}
\providecommand{\theoremname}{Theorem}

\begin{document}
\title{On the Fourier coefficients of word maps on unitary groups}
\dedicatory{In memory of Steve Zelditch}
\author{Nir Avni}
\address{Department of Mathematics, Northwestern University, 2033 Sheridan
Road, Evanston, IL 60208, USA}
\email{avni.nir@gmail.com}
\urladdr{https://sites.math.northwestern.edu/\textasciitilde nir/}
\author{Itay Glazer}
\address{Department of Mathematics, Northwestern University, 2033 Sheridan
Road, Evanston, IL 60208, USA}
\email{itayglazer@gmail.com}
\urladdr{https://sites.google.com/view/itay-glazer}
\begin{abstract}
Given a word $w(x_{1},\ldots,x_{r})$, i.e., an element in the free
group on $r$ elements, and an integer $d\geq1$, we study the characteristic
polynomial of the random matrix $w(X_{1},\ldots,X_{r})$, where $X_{i}$
are Haar-random independent $d\times d$ unitary matrices. If $c_{m}(X)$
denotes the $m$-th coefficient of the characteristic polynomial of
$X$, our main theorem implies that there is a positive constant $\epsilon(w)$,
depending only on $w$, such that
\[
\left|\mathbb{E}\left(c_{m}\left(w(X_{1},\ldots,X_{r})\right)\right)\right|\leq\binom{d}{m}^{1-\epsilon(w)},
\]
for every $d$ and every $1\leq m\leq d$.

Our main computational tool is the Weingarten Calculus, which allows
us to express integrals on unitary groups such as the expectation
above, as certain sums on symmetric groups. We exploit a hidden symmetry
to find cancellations in the sum expressing $\mathbb{E}\left(c_{m}(w)\right)$.
These cancellations, coming from averaging a Weingarten function over
cosets, follow from Schur's orthogonality relations.
\end{abstract}

\maketitle
\global\long\def\N{\mathbb{N}}%
\global\long\def\R{\mathbb{\mathbb{R}}}%
\global\long\def\Z{\mathbb{Z}}%
\global\long\def\val{\mathbb{\mathrm{val}}}%
\global\long\def\Qp{\mathbb{Q}_{p}}%
\global\long\def\Zp{\mathbb{\mathbb{Z}}_{p}}%
\global\long\def\ac{\mathbb{\mathrm{ac}}}%
\global\long\def\C{\mathbb{\mathbb{C}}}%
\global\long\def\Q{\mathbb{\mathbb{Q}}}%
\global\long\def\supp{\mathbb{\mathrm{supp}}}%
\global\long\def\VF{\mathbb{\mathrm{VF}}}%
\global\long\def\RF{\mathbb{\mathrm{RF}}}%
\global\long\def\VG{\mathbb{\mathrm{VG}}}%
\global\long\def\spec{\mathbb{\mathrm{Spec}}}%
\global\long\def\Ldp{\mathbb{\mathcal{L}_{\mathrm{DP}}}}%
\global\long\def\tr{\mathrm{tr}}%
\global\long\def\sgn{\mathrm{sgn}}%
\global\long\def\id{\mathrm{Id}}%
\global\long\def\Sym{\mathrm{Sym}}%
\global\long\def\U{\mathrm{U}}%
\global\long\def\Wg{\mathrm{Wg}}%

\raggedbottom

\section{\label{sec:Introduction}Introduction}

Let $w$ be a word on $r$ letters, i.e., an element in the free group
on the letters $x_{1},\ldots,x_{r}$. Let $X_{1},\ldots,X_{r}$ be
random $d\times d$ unitary matrices, chosen independently at random
according to the Haar probability measure, and consider the random
matrix $w(X_{1},\ldots,X_{r})$, obtained by substituting $X_{i}$
for $x_{i}$ in $w$. For example, if $w=x_{1}x_{2}x_{1}^{-1}x_{2}^{-1}$,
then $w(X_{1},X_{2})=X_{1}X_{2}X_{1}^{-1}X_{2}^{-1}$. In this paper,
we study the distribution of the characteristic polynomial of $w(X_{1},\ldots,X_{r})$.

To set notation, given a $d\times d$-matrix $A$ and $1\leq m\leq d$,
let $c_{m}(A)$ be the coefficient of $t^{d-m}$ in the characteristic
polynomial $\det(t\cdot\id-A)$ of $A$. Note that $c_{m}(A)=(-1)^{m}\tr\left(\bigwedge\nolimits ^{m}A\right)$,
where $\bigwedge^{m}A:\bigwedge^{m}\mathbb{C}^{d}\rightarrow\bigwedge^{m}\mathbb{C}^{d}$
is the $m$-th exterior power of $A$. If $A$ is unitary, all eigenvalues
have absolute value $1$, so we get the trivial bound $\left|c_{m}(A)\right|\leq\binom{d}{m}$.

Our main theorem is the following:
\begin{thm}
\label{thm:main_thm_fundemental_rep}For every non-trivial word $w\in F_{r}$,
there exists a constant $\epsilon(w)>0$ such that
\[
\mathbb{E}\left(\left|c_{m}\left(w(X_{1},\ldots,X_{r})\right)\right|^{2}\right)\leq\binom{d}{m}^{2(1-\epsilon(w))},
\]
for every $d$ and every $1\leq m\leq d$. In particular, we have
\[
\mathbb{E}\left(\left|c_{m}\left(w(X_{1},\ldots,X_{r})\right)\right|\right)\leq\binom{d}{m}^{1-\epsilon(w)}.
\]
\end{thm}

\begin{rem}
\label{rem:main_explicit_dependence}~
\begin{enumerate}
\item In the proof of Theorem \ref{thm:main_thm_fundemental_rep}, we show
that, if the length of $w$ is $\ell$ and $d\geq(25\ell)^{7\ell}$,
then one can take $\epsilon(w)=\frac{1}{72}\left(25\ell\right)^{-2\ell}$.
We believe $\epsilon(w)^{-1}$ can be taken to be a polynomial in
$\ell$, for $d\gg_{\ell}1$.
\item On the other hand, it follows from \cite[Theorem 5.2]{ET15} that,
for a fixed $d$, one has to take $\epsilon(w)\lesssim e^{-\sqrt{\ell}}$,
for some arbitrarily long words, even for $m=1$. 
\end{enumerate}
\end{rem}

Theorem \ref{thm:main_thm_fundemental_rep} relies on the following:
\begin{thm}
\label{thm:abs_value_of_trace}For every $m,\ell\in\N$, every $d\geq m\ell$,
and every word $w\in F_{r}$ of length $\ell$, one has: 
\begin{equation}
\mathbb{E}\left(\left|c_{m}\left(w(X_{1},\ldots,X_{r})\right)\right|^{2}\right)\leq(22\ell)^{m\ell}.\label{eq:bounds on fundemental weights}
\end{equation}
In particular, if $d\geq(22\ell)^{\ell}m$, we have 
\[
\mathbb{E}\left(\left|c_{m}\left(w(X_{1},\ldots,X_{r})\right)\right|^{2}\right)\leq\binom{d}{m}.
\]
\end{thm}

In addition, we show that similar bounds hold for symmetric powers:
\begin{thm}
\label{thm:main thm symmetric powers}For every $\ell\in\N$, every
$d\geq m\ell$, and every word $w\in F_{r}$ of length $\ell$, one
has: 
\[
\mathbb{E}\left(\left|\tr\left(\Sym^{m}w(X_{1},\ldots,X_{r})\right)\right|^{2}\right)\leq(16\ell)^{m\ell}.
\]
In particular, if $d\geq(16\ell)^{\ell}m$, we have
\[
\mathbb{E}\left(\left|\tr\left(\Sym^{m}w(X_{1},\ldots,X_{r})\right)\right|^{2}\right)\leq\binom{d+m-1}{m}=\dim\Sym^{m}\mathbb{C}^{d},
\]
and by the Cauchy\textendash Schwarz inequality, 
\[
\left|\mathbb{E}\left(\tr\left(\Sym^{m}w(X_{1},\ldots,X_{r})\right)\right)\right|\leq\left(\dim\Sym^{m}\mathbb{C}^{d}\right)^{\frac{1}{2}}.
\]
\end{thm}

\begin{rem}
Theorem \ref{thm:main thm symmetric powers} is an analogue of Theorem
\ref{thm:abs_value_of_trace}. It is also an analogue of Theorem \ref{thm:main_thm_fundemental_rep}
for $m$ at most linear in $d$. Contrary to exterior powers, the
methods of this paper are insufficient for finding bounds similar
to Theorem \ref{thm:main_thm_fundemental_rep} for $\left|\mathbb{E}\left(\tr\left(\Sym^{m}w(X_{1},\ldots,X_{r})\right)\right)\right|$,
in the regime where $m$ is superlinear in $d$. 
\end{rem}

\subsection{\label{subsec:Related-work}Related work }

Word maps on unitary groups and their eigenvalues have been extensively
studied in the past few decades. 

The case $w=x$, namely, the study of a Haar-random unitary matrix
$X$, also known as the Circular Unitary Ensemble (CUE), is an important
object of study in random matrix theory (see e.g.~\cite{AGZ10,Mec19}
and the references within). The joint density of the eigenvalues of
$X$ is given by the Weyl Integration Formula \cite{Wey39}. Schur's
orthogonality relations immediately imply that $\mathbb{E}\left(\left|c_{m}(X)\right|^{2}\right)=1$
for all $1\leq m\leq d$. Various other properties of the characteristic
polynomial of a random unitary matrix $X$ have been extensively studied
(see e.g.~\cite{KS00,HKO01,CFKRS03,DG04,BG06,BHNY08,ABB17,CMN18,PZ18}). 

Diaconis and Shahshahani \cite{DS94} have shown that, for a fixed
$m\in\N$, the sequence of random variables $\tr(X),\tr(X^{2}),\ldots,\tr(X^{m})$
converges in distribution, as $d\rightarrow\infty$, to a sequence
of independent complex normal random variables. For the proof, which
relies on the moment method, they computed the joint moments of those
random variables and showed that
\begin{equation}
\mathbb{E}\left(\prod_{j=1}^{m}\tr(X^{j})^{a_{j}}\tr(\overline{X}^{j})^{b_{j}}\right)=\delta_{a,b}\prod_{j=1}^{m}j^{a_{j}}a_{j}!,\label{eq:DS moments of trace}
\end{equation}
for $d\geq\sum_{j=1}^{m}(a_{j}+b_{j})j$. The rate of convergence
was later shown to be super-exponential by Johansson \cite{Joh97}.

When $w=x^{\ell}$, (\ref{eq:DS moments of trace}) gives a formula
for the moments of traces, and one can use Newton's identities relating
elementary symmetric polynomials and power sums, to deduce that
\[
\mathbb{E}\left(\left|c_{m}(X^{\ell})\right|^{2}\right)=\mathbb{E}\left(\left|\tr\left(\Sym^{m}w\right)\right|^{2}\right)=\binom{\ell+m-1}{m},
\]
for $d\geq2m\ell$ (see Appendix \ref{sec:Fourier-coefficients-of power word}).
In \cite{Rai97,Rai03}, Rains partially extended (\ref{eq:DS moments of trace})
for small $d$ and gave an explicit formula for the joint density
of the eigenvalues of $X^{\ell}$ (see \cite[Theorem 1.3]{Rai03}).

We now move to general words $w\in F_{r}$. The case $m=1$, namely,
the asymptotics as $d\rightarrow\infty$ of the distribution of the
random variable $\tr\left(w(X_{1},\ldots,X_{r})\right)$, was studied
in the context of Voiculescu's free probability (see e.g.~\cite{DNV92,MS17}).
In particular, in \cite{Voi91,Rud06,MSS07} it was shown that, for
a fixed $w\in F_{r}$, the sequence of random variables $\tr\left(w(X_{1},\ldots,X_{r})\right)$,
for $d=1,2,\ldots$, converges in distribution, as $d\rightarrow\infty$,
to a complex normal random variable (with suitable normalization).
As a direct consequence, for a fixed $m\in\N$, the random variables
$c_{m}\left(w(X_{1},\ldots,X_{r})\right)$ converge, as $d\rightarrow\infty$,
to a certain explicit polynomial of Gaussian random variables. This
is done in Appendix \ref{sec:Fourier-coefficients-of power word},
Corollary \ref{cor:moment method}, following \cite{DG04}.

In \cite{MP19}, Magee and Puder have shown that $\mathbb{E}\left(\tr\left(w(X_{1},\ldots,X_{r})\right)\right)$
coincides with a rational function of $d$, if $d$ is sufficiently
large, and bounded its degree in terms of the commutator length of
$w$. They also found a geometric interpretation for the coefficients
of the expansion of that rational function as a power series in $d^{-1}$,
see \cite[Corollaries 1.8 and 1.11]{MP19}. See \cite{Bro} for additional
work in this direction.

\subsection{Ideas of proofs}

With a few exceptions, the results stated in $\mathsection$\ref{subsec:Related-work}
are asymptotic in $d$, but not uniform in both $m$ and $d$. We
will try to explain some of the challenges in proving results that
are uniform in $m$, while explaining the idea of the proof of Theorem
\ref{thm:main_thm_fundemental_rep}.

Our main tool (which is also used in the papers \cite{Rud06,MSS07,MP19}
above) to study integrals on unitary groups is the \emph{Weingarten
Calculus} (\cite{Wein78,Col03,CS06}). Roughly speaking, the Weingarten
Calculus utilizes the Schur\textendash Weyl Duality to express integrals
on unitary groups as sums of so called Weingarten functions over symmetric
groups. In our case, in order to prove Theorem \ref{thm:main_thm_fundemental_rep},
we need to estimate the integral
\begin{equation}
\mathbb{E}\left(\left|c_{m}(w)\right|^{2}\right)=\int_{\U_{d}^{r}}\left|\tr\left(\bigwedge\nolimits ^{m}w(X_{1},\ldots,X_{r})\right)\right|^{2}dX_{1}\ldots dX_{r}.\label{eq:integral formula for main thm}
\end{equation}
Using Weingarten calculus (Theorem \ref{thm:Weingarten calculus}),
we express (\ref{eq:integral formula for main thm}) as a finite sum
\begin{equation}
\sum_{(\pi_{1},\ldots,\pi_{2r})\in\prod_{i=1}^{2r}S_{m\ell_{i}}}F(\pi_{1},\ldots,\pi_{2r})\prod_{i=1}^{r}\Wg_{d}^{(i)}(\pi_{i}\pi_{i+r}^{-1}),\label{eq:reduction to finite sum}
\end{equation}
where $\ell_{1},\ldots,\ell_{2r}\in\mathbb{N}$ and $F:\prod_{i=1}^{2r}S_{m\ell_{i}}\rightarrow\mathbb{Z}$
are related to combinatorial properties of $w$, and each $\Wg_{d}^{(i)}:S_{m\ell_{i}}\rightarrow\R$
is a \emph{Weingarten function }(see Definition \ref{def:Weingarten function}).
There are two main difficulties when dealing with sums such as (\ref{eq:reduction to finite sum})
in the region when $m$ is unbounded:
\begin{enumerate}
\item While the asymptotics of Weingarten functions $\Wg_{d}:S_{m}\rightarrow\R$
are well understood when $d\gg m$ (see \cite[Section 2.2]{Col03}
and \cite[Therem 1.1]{CM17}), much less is known in the regime where
$m$ is comparable with $d$. 
\item Even if we have a good understanding of a single Weingarten function,
the number of summands in (\ref{eq:reduction to finite sum}) is large
and it is not enough to bound each individual Weingarten function.
\end{enumerate}
Luckily, there are plenty of cancellations in the sum (\ref{eq:reduction to finite sum}).
To understand these cancellations, we identify a symmetry of (\ref{eq:reduction to finite sum}).
More precisely, we find a group $H$ acting on $\prod_{i=1}^{2r}S_{m\ell_{i}}$
such that $F$ is equivariant with respect to $H$, and such that
the contribution of any $H$-orbit to the sum (\ref{eq:reduction to finite sum})
is a product of terms, each of which has the form
\begin{equation}
\frac{1}{m!^{2\ell_{i}}}\sum_{h,h'\in S_{m}^{\ell_{i}}}\sgn\left(hh'\right)\Wg_{d}^{(i)}\left(h'\pi_{i}h\pi_{i+r}^{-1}\right),\label{eq:twisted.average}
\end{equation}
where $\sgn(x)$ is the sign of $x$ and the sum is over the Young
subgroup $S_{m}^{\ell_{i}}\subseteq S_{m\ell_{i}}$, see Corollary
\ref{cor:analyzing H-orbits}.

Weingarten functions are class functions, so they are linear combinations
of irreducible characters of $S_{m\ell_{i}}$. Explicitly, we have
(see \cite[Eq. (13)]{CS06}):
\begin{equation}
\Wg_{d}^{(i)}(\sigma)=\frac{1}{(m\ell_{i})!^{2}}\sum_{\lambda\vdash m\ell_{i},\ell(\lambda)\leq d}\frac{\chi_{\lambda}(1)^{2}}{\rho_{\lambda}(1)}\chi_{\lambda}(\sigma),\,\,\,\,\,\,\,\,\,\,\,\,\,\sigma\in S_{m\ell_{i}},\label{eq:formula for Weingarten}
\end{equation}
where each $\lambda$ is a partition of $m\ell_{i}$ with at most
$d$ parts and $\chi_{\lambda}$ and $\rho_{\lambda}$ are the corresponding
irreducible characters of $S_{m\ell_{i}}$ and $\U_{d}$, respectively.
The cancellations that we get in the sum (\ref{eq:twisted.average})
come from averaging irreducible characters of $S_{m\ell_{i}}$ over
$S_{m}^{\ell_{i}}$-cosets. $S_{m}^{\ell_{i}}$ is a large subgroup
of $S_{m\ell_{i}}$, so these cancellations will be significant as
well. For example, all terms in (\ref{eq:formula for Weingarten})
for which $\lambda$ has more than $\ell_{i}$ columns vanish. See
Lemmas \ref{lem:primitive character bounds} and \ref{lem:bounds on double cosets}
for the precise bounds. 

After we bound the average contribution of each $H$-orbit in the
sum (\ref{eq:reduction to finite sum}) by a function $C(m,d,w)$,
we bound (\ref{eq:reduction to finite sum}) by $\left|Z\right|\cdot C(m,d,w)$
for some finite set $Z$. This becomes a counting problem, which we
solve in $\mathsection$\ref{sec:Estimates-on}, see Proposition \ref{prop:estimates on Z}. 

The proof of Theorem \ref{thm:main_thm_fundemental_rep} occupies
Sections \ref{sec:Reduction-of-Theorem}, \ref{sec:Estimating-the-contribution},
\ref{sec:Estimates-on} and \ref{sec:Proof-of-main theorems}. Since
the combinatorics of general words is a bit complicated, we prove
a simplified version of Theorem \ref{thm:abs_value_of_trace} for
the special case of the Engel word $[[x,y],y]$ in $\mathsection$\ref{sec:The-Engel-word}.
The proof for this special case contains the main ideas of the paper,
while being easier to understand. 

\subsection{Further discussion and some open questions}

The results of this paper fit in the larger framework of the study
of word measures and their Fourier coefficients. 

Let $G$ be a compact group, and let $\mu_{G}$ be the Haar probability
measure on $G$. To each word $w(x_{1},\ldots,x_{r})\in F_{r}$ we
associate the corresponding \emph{word map} $w_{G}:G^{r}\rightarrow G$,
defined by $(g_{1},\ldots,g_{r})\mapsto w(g_{1},\ldots,g_{r})$. The
pushforward measure $(w_{G})_{*}(\mu_{G}^{r})$ is called the \emph{word
measure} $\tau_{w,G}$ associated with $w$ and $G$. Let $\mathrm{Irr}(G)$
be the set of irreducible characters of $G$. The \emph{Fourier coefficient}
of $\tau_{w,G}$ at $\rho\in\mathrm{Irr}(G)$ is
\begin{equation}
a_{w,G,\rho}:=\int_{G^{r}}\rho(w(x_{1},\ldots,x_{r}))\mu_{G}^{r}=\int_{G}\rho(y)\tau_{w,G}.\label{eq:Fourier coefficients}
\end{equation}
If $w\neq1$ and $G$ is a compact semisimple Lie group, then by Borel's
theorem \cite{Bor83}, the map $w_{G}:G^{r}\rightarrow G$ is a submersion
outside a proper subvariety in $G^{r}$. It follows that $\tau_{w,G}$
is absolutely continuous with respect to $\mu_{G}$ and, therefore,
$\tau_{w,G}=f_{w,G}\cdot$$\mu_{G}$, where $f_{w,G}\in L^{1}(G)$
is the Radon\textendash Nikodym density. In this case, $f_{w,G}=\sum_{\rho\in\mathrm{Irr}(G)}\overline{a_{w,G,\rho}}\cdot\rho$. 

In \cite[Theorem 4]{LST19}, Larsen, Shalev, and Tiep proved uniform
$L^{\infty}$-mixing time for convolutions of word measures on sufficiently
large finite simple groups. From this, the following can be deduced: 
\begin{thm}
\label{thm:LST theorem}For every $w\in F_{r}$, there exists $N(w)\in\N$
such that if $G$ is a finite simple group with at least $N(w)$ elements,
then
\begin{equation}
\left|a_{w,G,\rho}\right|\leq(\dim\rho)^{1-\epsilon(w)},\label{eq:Fourier decay for finite simple groups}
\end{equation}
for $\epsilon(w)=C\cdot\ell(w)^{-4}$ and some absolute constant $C$.
\end{thm}

The proof of Theorem \ref{thm:LST theorem} is given at the end of
$\mathsection$\ref{sec:Proof-of-main theorems}.

We believe that a similar statement should be true for compact semisimple
Lie groups. 
\begin{conjecture}
\label{conj:Fourier coefficients of compact Lie groups}For every
$1\neq w\in F_{r}$, there exists $\epsilon(w)>0$ such that, for
every compact connected semisimple Lie group $G$ and every $\rho\in\mathrm{Irr}(G)$,
\[
\left|a_{w,G,\rho}\right|\leq(\dim\rho)^{1-\epsilon(w)}.
\]
\end{conjecture}

It is natural to estimate $\epsilon(w)$ in terms of the length $\ell(w)$
of the word $w$. For simple groups of bounded rank, Item (2) of Remark
\ref{rem:main_explicit_dependence} (i.e.~\cite[Theorem 5.2]{ET15})
shows that there are arbitrarily long words $w$ for which $\epsilon(w)$
cannot be larger than $e^{-\sqrt{\ell(w)}}$. However, we believe
that better Fourier decay can be achieved for the high rank case. 
\begin{question}
Can one take $\epsilon(w)$ to be a polynomial in $\ell(w)$, if $\mathrm{rk}(G)\gg_{\ell(w)}1$? 
\end{question}

Theorem \ref{thm:main_thm_fundemental_rep} gives evidence to Conjecture
\ref{conj:Fourier coefficients of compact Lie groups} for $G=\mathrm{SU}_{d}$
and the collection of \emph{fundamental representations} $\left\{ \bigwedge^{m}\C^{d}\right\} _{m=1}^{d}$.
Indeed, for every $\rho\in\mathrm{Irr}(\mathrm{U}_{d})$, since $\left|\rho(\lambda A)\right|=\left|\rho(A)\right|$
for $A\in\mathrm{SU}_{d}$ and $\lambda\in\U_{1}$, and since $\mu_{\U_{d}}$
is the pushforward of $\mu_{\U_{1}}\times\mu_{\mathrm{SU}_{d}}$ by
the multiplication map $(\lambda,A)\mapsto\lambda A$, we have,
\begin{align}
 & \left|a_{w,\mathrm{SU}_{d},\rho}\right|^{2}\leq\mathbb{E}_{X_{1},\ldots,X_{r}\in\mathrm{SU}_{d}}\left(\left|\rho\left(w(X_{1},\ldots,X_{r})\right)\right|^{2}\right)\label{eq:reduction to Ud}\\
= & \mathbb{E}_{(\lambda_{1},X_{1}),\ldots,(\lambda_{r},X_{r})\in\mathrm{SU}_{d}\times\U_{1}}\left(\left|\rho\left(w(\lambda_{1},\ldots,\lambda_{r})w(X_{1},\ldots,X_{r})\right)\right|^{2}\right)\nonumber \\
= & \mathbb{E}_{(\lambda_{1},X_{1}),\ldots,(\lambda_{r},X_{r})\in\mathrm{SU}_{d}\times\U_{1}}\left(\left|\rho\left(w(\lambda_{1}X_{1},\ldots,\lambda_{r}X_{r})\right)\right|^{2}\right)=\mathbb{E}_{\U_{d}}\left(\left|\rho\left(w(X_{1},\ldots,X_{r})\right)\right|^{2}\right),\nonumber 
\end{align}
Theorem \ref{thm:main thm symmetric powers} deals with another family
of irreducible representations $\left\{ \Sym^{m}\C^{d}\right\} _{m=1}^{\left\lfloor d/(16\ell)^{\ell}\right\rfloor }$,
giving further evidence for Conjecture \ref{conj:Fourier coefficients of compact Lie groups}.

Verifying Conjecture \ref{conj:Fourier coefficients of compact Lie groups}
will imply that, for every word $w$, the random walks induced by
the collection of measures $\left\{ \tau_{w,G}\right\} _{G}$, where
$G$ runs over all compact connected simple Lie groups, admit a uniform
$L^{\infty}$-mixing time. Namely, using \cite[Theorem 1]{GLM12},
it will show the existence of $t(w)\in\N$ such that 
\begin{equation}
\left\Vert \frac{\tau_{w,G}^{*t(w)}}{\mu_{G}}-1\right\Vert _{\infty}<1/2,\label{eq:uniform mixing time}
\end{equation}
for every compact connected simple Lie group $G$. By the above discussion,
$t(w)$ grows at least exponentially with $\sqrt{\ell(w)}$ under
no restriction on the rank. If the condition (\ref{eq:uniform mixing time})
is replaced by the condition that $\tau_{w,G}^{*t(w)}$ has bounded
density, one might hope for polynomial bounds. 
\begin{question}
\label{que:words have bounded density}Let $1\neq w\in F_{r}$. Can
one find $t(w)\in\N$ such that for every compact connected semisimple
Lie group $G$, $\tau_{w,G}^{*t(w)}$ has bounded density with respect
to $\mu_{G}$? can $t(w)$ be chosen to have polynomial dependence
on $\ell(w)$?
\end{question}

Question \ref{que:words have bounded density} can be seen as an analytic
specialization of a geometric phenomenon. Let $\varphi:X\rightarrow Y$
be a polynomial map between smooth $\Q$-varieties. We say that $\varphi$
is \emph{(FRS)} if it is flat and its fibers all have rational singularities.
In \cite[Theorem 3.4]{AA16}, Aizenbud and the first author showed
that if $\varphi$ is (FRS), then for every non-Archimedean local
field $F$ and every smooth, compactly supported measure $\mu$ on
$X(F)$, the pushforward $\varphi_{*}\mu$ has bounded density. This
result was extended in \cite{Rei} to the Archimedean case, $F=\R$
or $\C$, and, moreover, if one runs over a large enough family of
local fields, the condition of (FRS) is in fact necessary as well
for the densities of pushforwards to be bounded (see \cite[Theorem 3.4]{AA16}
and \cite[Corollary 6.2]{GHS}).

To rephrase Question \ref{que:words have bounded density} in geometric
term, we further need the following notion from \cite{GH19,GH21}.
\begin{defn}[{\cite[Definition 1.1]{GH19}}]
\label{def:convolution}Let $\varphi:X\rightarrow G$ and $\psi:Y\rightarrow G$
be morphisms from algebraic varieties $X,Y$ to an algebraic group
$G$. We define their \emph{convolution} by 
\[
\varphi*\psi:X\times Y\rightarrow G,\text{ }(x,y)\mapsto\varphi(x)\cdot\psi(y).
\]
We denote by $\varphi^{*k}:X^{k}\rightarrow G$ the \emph{$k$-fold
convolution} of $\varphi$ with itself. 
\end{defn}

Based on the above discussion, a positive answer to the following
question will answer Question \ref{que:words have bounded density}
positively. 
\begin{question}[{\cite[Question 1.15]{GHb}}]
\label{que:conjecture on (FRS) and flatness}Can we find $\alpha,C>0$
such that, for every $w\in F_{r}$ of length $\ell$ and every simple
algebraic group $G$, the word map $w_{G}^{*C\ell^{\alpha}}$ is (FRS)? 
\end{question}

In \cite{GH19,GH21}, Yotam Hendel and the second author and have
shown that any dominant map $\varphi:X\rightarrow G$ from a smooth
variety to a connected algebraic group becomes (FRS) after sufficiently
many self-convolutions. Concrete bounds were given in \cite[Corollary 1.9]{GHS}.
Based on these results, we prove Conjecture \ref{conj:Fourier coefficients of compact Lie groups}
and answer Question \ref{que:words have bounded density} for the
bounded rank case (see Proposition \ref{prop:equidistribution- finite rank}).

To conclude the discussion, we remark that a positive answer for Question
\ref{que:conjecture on (FRS) and flatness} will answer Question \ref{que:words have bounded density}
for compact semisimple $p$-adic groups as well. A significant progress
in this direction was done in the work \cite{GHb}, by Yotam Hendel
and the second author, where singularities of word maps on semisimple
Lie algebras and algebraic groups were studied.

\subsection{Conventions and notations}
\begin{enumerate}
\item We denote the set $\{1,\ldots,N\}$ by $[N]$.
\item For a finite set $X$, we denote the symmetric group on $X$ by $\Sym(X)$
and the space of functions $f:X\rightarrow\C$ by $\C[X]$ . 
\item We write $(-1)^{\sigma}$ for the sign of a permutation $\sigma$. 
\item For a group $G$, a representation is a pair $(\pi,V)$, with $\pi:G\rightarrow\mathrm{GL}(V)$
a homomorphism. We denote the character of $(\pi,V)$ by $\chi_{\pi}$
and denote its dual by $(\pi^{\vee},V^{\vee})$.
\end{enumerate}
\begin{acknowledgement*}
We thank Rami Aizenbud, Yotam Hendel, Michael Larsen, Michael Magee,
Doron Puder, Yotam Shomroni, Ofer Zeitouni and Steve Zelditch for
useful conversations. We thank the referees for their useful comments
and for improving the readability of the paper. NA was supported by
NSF grant DMS\textendash 1902041, IG was supported by AMS\textendash Simons
travel grant, and both of us were supported by BSF grant 2018201.
\end{acknowledgement*}

\section{Preliminaries}

\subsection{\label{subsec:Some-facts-in rep theory}Some facts in representation
theory}

For a compact group $G$, we denote the set of irreducible complex
characters of $G$ by $\mathrm{Irr}(G)$. Given a subgroup $H\leq G$
and a character $\chi\in\mathrm{Irr}(H)$, we denote the induction
of $\chi$ to $G$ by $\mathrm{Ind}_{H}^{G}\chi$. We normalize the
Haar measure to be a probability measure and denote the expectation
with respect to the Haar measure by $\mathbb{E}$. The standard inner
product on functions on $G$ is $\langle f_{1},f_{2}\rangle_{G}=\mathbb{E}f_{1}\overline{f_{2}}$. 

\subsubsection{\label{subsec:Representation-theory-of symmetric}Representation
theory of the symmetric group}

Given $m\in\N$, a \emph{partition} of $m$ is a non-increasing sequence
$\lambda=(\lambda_{1},..,\lambda_{k})$ of non-negative integers that
sum to $m$. In this case, we write $\lambda\vdash m$. Two partitions
are equivalent if they differ only by a string of $0$'s at the end.
A partition $\lambda=(\lambda_{1},..,\lambda_{k})$, with $\lambda_{k}>0$,
is graphically encoded by a \emph{Young diagram}, which is a finite
collection of boxes (or cells) arranged in $k$ left-justified rows,
where the $j$-th row has $\lambda_{j}$ boxes. The \emph{length}
$\ell(\lambda)$ of a partition $\lambda\vdash m$ is the number of
non-zero parts $\lambda_{i}$ or equivalently the number of rows in
the corresponding Young diagram.

The irreducible representations of $S_{m}$ are in bijection with
partitions $\lambda\vdash m$. We write $\chi_{\lambda}\in\mathrm{Irr}(S_{m})$
for the corresponding character. For each cell $(i,j)$ in the Young
diagram of $\lambda$, the \emph{hook length} $h_{\lambda}(i,j)$
is the number of cells $(a,b)$ in the Young diagram of $\lambda$
such that either $a=i$ and $b\geq j$, or $a\geq i$ and $b=j$.
The hook-length formula states that
\begin{equation}
\chi_{\lambda}(1)=\frac{m!}{\prod_{(i,j)\in\lambda}h_{\lambda}(i,j)}.\label{eq:Hook length}
\end{equation}

\begin{defn}
\label{def:=00005Cmu extansion of a Young diagram}~
\begin{enumerate}
\item Fix a Young Diagram $\lambda$ and let $n\in\mathbb{N}$. An \emph{$n$-expansion}
of $\lambda$ is any Young diagram obtained by adding $n$ boxes to
$\lambda$ in such a way that no two boxes are added in the same column.
\item Given a partition $\lambda=(\lambda_{1},\ldots,\lambda_{l_{1}})\vdash k$
and a partition $\mu=(\mu_{1},\ldots,\mu_{l_{2}})\vdash l$, a $\mu$-expansion
of $\,\lambda$ is defined to be a $\mu_{l_{2}}$-expansion of a $\mu_{l_{2}-1}$-expansion
of a $\cdots$ of a $\mu_{1}$-expansion of the Young diagram of $\lambda$.
For a $\mu$-expansion of $\lambda$, we label the boxes added in
the $\mu_{l_{j}}$-expansion by the number $j$ and order the boxes
lexicographically by their position, first from top to bottom and
then from right to left. We say that a $\mu$-expansion of $\lambda$
is \emph{strict} if, for every $p\in\{1,\ldots,l_{2}-1\}$ and every
box $t$, the number of boxes coming before $t$ that are labeled
$p$ is greater than or equal to the number of boxes coming before
$t$ that are labeled $(p+1)$. 
\end{enumerate}
\end{defn}

\begin{thm}[{Littlewood\textendash Richardson rule, \cite[I.9]{Mac95}}]
\label{thm:Littlewood-Richardson}Let $\lambda\vdash k$ and $\mu\vdash m$.
Then,
\[
\mathrm{Ind}_{\mathrm{S}_{k}\times S_{m}}^{S_{k+m}}(\chi_{\lambda}\otimes\chi_{\mu})=\bigoplus_{\nu\vdash k+m}N_{\lambda\mu\nu}\chi_{\nu},
\]
where $N_{\lambda\mu\nu}$ is the number of strict $\mu$-expansions
of $\lambda$ that are a Young diagram of the partition $\nu$.
\end{thm}

We will need the following consequence of Theorem \ref{thm:Littlewood-Richardson}.
\begin{lem}
\label{lem:description of littlewood richardson for Young subgroups}Let
$l\in\Z_{\geq2}$ and identify $S_{m}^{l}$ with its image in the
standard embedding $S_{m}^{l}\hookrightarrow S_{ml}$. Then, each
$\chi_{\nu}\in\mathrm{Irr}(S_{ml})$ appearing in $\mathrm{Ind}_{\mathrm{S}_{m}^{l}}^{S_{ml}}(1)$
(resp., $\mathrm{Ind}_{\mathrm{S}_{m}^{l}}^{S_{ml}}(\sgn)$) corresponds
to a partition $\nu\vdash ml$ with at most $l$ rows (resp., $l$
columns).
\end{lem}

\begin{proof}
We prove the statement for the trivial representation 1 by induction
on $l$. The proof for $\sgn$ is similar. The character $1$ of $S_{m}$
corresponds to the partition $\lambda$ consisting of one row of length
$m$. By the induction hypothesis, we may assume that $\mathrm{Ind}_{\mathrm{S}_{m}^{j}}^{S_{mj}}(1)=\bigoplus_{\mu\vdash mj}m_{\mu}\chi_{\mu}$,
with $m_{\mu}>0$ only if $\mu$ has at most $j$ rows, for all $j<l$.
Hence we can write
\begin{equation}
\mathrm{Ind}_{\mathrm{S}_{m}^{l}}^{S_{ml}}(1)=\mathrm{Ind}_{\mathrm{S}_{m(l-1)}\times S_{m}}^{S_{ml}}(\mathrm{Ind}_{\mathrm{S}_{m}^{l-1}}^{\mathrm{S}_{m(l-1)}}(1)\otimes1)=\bigoplus_{\mu\vdash m(l-1)}m_{\mu}\mathrm{Ind}_{\mathrm{S}_{m(l-1)}\times S_{m}}^{S_{ml}}(\chi_{\mu}\otimes1).\label{eq:double induction}
\end{equation}
By Theorem \ref{thm:Littlewood-Richardson} and since a strict $\lambda$-expansion
of $\mu$ increases the number of rows by at most one, the lemma follows. 
\end{proof}

\subsubsection{\label{subsec:Representation-theory-of unitary }Representation theory
of the unitary group}

The irreducible representations of $\U_{d}$ can be identified with
the irreducible rational representations of $\mathrm{GL}_{d}(\C)$.
More precisely, the restriction map $\rho\mapsto\rho|_{\U_{d}}$ induces
a bijection $\mathrm{Irr}(\mathrm{GL}_{d}(\C))\rightarrow\mathrm{Irr}(\U_{d})$.
Moreover, the set $\mathrm{Irr}(\U_{d})$ is in bijection with the
set $\Lambda$ of dominant weights,
\[
\Lambda:=\left\{ (\lambda_{1},\ldots,\lambda_{d}):\lambda_{1}\geq\ldots\geq\lambda_{d},\,\,\lambda_{i}\in\Z\right\} .
\]
We denote the representation corresponding to $\lambda\in\Lambda$
by $(\rho_{\lambda},V_{\lambda})$. The irreducible representations
\[
\C^{d},\bigwedge\nolimits ^{2}\C^{d},\ldots,\bigwedge\nolimits ^{d}\C^{d},
\]
are called \emph{the fundamental representations}, and we have $\bigwedge\nolimits ^{m}\C^{d}\simeq V_{(1,\ldots,1,0,\ldots,0)}$,
with $1$ appearing $m$ times. In particular, the standard representation
$\C^{d}$ is $V_{(1,0,\ldots,0)}$. Note that $\bigwedge^{d}\C^{d}$
is the determinant representation $\chi_{\det}$. We identify a weight
$\lambda\in\Lambda$ such that $\lambda_{d}\geq0$ with a partition
$(\lambda_{1},\ldots,\lambda_{d})$. 
\begin{rem}[{\cite[I.6, Exc. 6.4]{FH91}}]
\label{rem:formula for dimension of rep}For each $\lambda=(\lambda_{1},\ldots,\lambda_{d})\vdash m$,
\begin{equation}
\rho_{\lambda}(1)=\frac{\chi_{\lambda}(1)\cdot\prod_{(i,j)\in\lambda}(d+j-i)}{m!},\label{eq:formula for dimension}
\end{equation}
where $(i,j)$ are the coordinates of the cells in the Young diagram
with shape $\lambda$. 
\end{rem}

Given $\lambda,\mu\in\Lambda$, the irreducible subrepresentations
of $\rho_{\lambda}\otimes\rho_{\mu}$ are determined by the Littlewood\textendash Richardson
rule as follows. 
\begin{thm}[{Littlewood\textendash Richardson rule, see e.g.~\cite[I.6, Eq. (6.7)]{FH91}}]
\label{thm:Littlewood-Richardson-unitary}Let $\lambda,\mu\in\Lambda$
and suppose that $\lambda_{d},\mu_{d}\geq0$. Let $N_{\lambda\mu\nu}$
be the coefficients from Theorem \ref{thm:Littlewood-Richardson}.
Then:
\[
\rho_{\lambda}\otimes\rho_{\mu}=\bigoplus_{\nu:\nu_{d}\geq0}N_{\lambda\mu\nu}\rho_{\nu}.
\]
\end{thm}

\begin{rem}
\label{rem:generalized Littlewood-Richardson}For $\lambda,\mu\in\Lambda$,
set $\widetilde{\lambda}:=\lambda-(\lambda_{d},\ldots,\lambda_{d})$
and $\widetilde{\mu}:=\mu-(\mu_{d},\ldots,\mu_{d})$. Then $\rho_{\lambda}=\chi_{\det}^{\lambda_{d}}\cdot\rho_{\widetilde{\lambda}}$
and $\rho_{\mu}=\chi_{\det}^{\mu_{d}}\cdot\rho_{\widetilde{\mu}}$,
and hence by Theorem \ref{thm:Littlewood-Richardson-unitary}, one
has: 
\begin{equation}
\rho_{\lambda}\otimes\rho_{\mu}=\chi_{\det}^{\lambda_{d}+\mu_{d}}\rho_{\widetilde{\lambda}}\otimes\rho_{\widetilde{\mu}}=\chi_{\det}^{\lambda_{d}+\mu_{d}}\bigoplus_{\nu}N_{\widetilde{\lambda}\widetilde{\mu}\nu}\rho_{\nu}.\label{eq:generalized Littlewood-Richardson}
\end{equation}
\end{rem}

\subsubsection{\label{subsec:Averaging-characters-over}Averaging characters over
cosets}
\begin{lem}
\label{lem:primitive character bounds}Let $G$ be a finite group,
let $(\pi,V)$ be an irreducible representation of $G$, let $H\leq G$
be a subgroup, and let $\lambda$ be any one-dimensional character
of $H$. Then, for every $g\in G$, 
\[
\left|\frac{1}{\left|H\right|}\sum_{h\in H}\lambda^{-1}(h)\chi_{\pi}(gh)\right|\leq\left\langle \chi_{\pi}|_{H},\lambda\right\rangle _{H}.
\]
In particular, if $\left\langle \chi_{\pi}|_{H},\lambda\right\rangle _{H}=0$,
then $\sum_{h\in H}\lambda^{-1}(h)\chi_{\pi}(gh)=0$. 
\end{lem}

\begin{proof}
Write $\pi|_{H}=\bigoplus_{i=1}^{\widetilde{N}}\pi_{i}$ with each
$(\pi_{i},V_{i})$ an irreducible representation of $H$. For each
$i$ and $h'\in H$, 
\begin{align*}
\left(\sum_{h\in H}\lambda^{-1}(h)\pi_{i}(h)\right)\pi_{i}(h') & =\sum_{h\in H}\lambda^{-1}(h)\pi_{i}(hh')=\sum_{h\in H}\lambda^{-1}(hh'^{-1})\pi_{i}(h)\\
 & =\sum_{h\in H}\lambda^{-1}(h'^{-1}h)\pi_{i}(h)=\sum_{h\in H}\lambda^{-1}(h)\pi_{i}(h'h)=\pi_{i}(h')\left(\sum_{h\in H}\lambda^{-1}(h)\pi_{i}(h)\right).
\end{align*}
By Schur's lemma, $\sum_{h\in H}\lambda^{-1}(h)\pi_{i}(h)$ is a scalar
matrix $\alpha\cdot I_{V_{i}}$, for some $\alpha\in\C$. Hence,
\begin{equation}
\alpha\cdot\chi_{\pi_{i}}(1)=\tr\left(\sum_{h\in H}\lambda^{-1}(h)\pi_{i}(h)\right)=\sum_{h}\lambda^{-1}(h)\chi_{\pi_{i}}(h)=\begin{cases}
\left|H\right| & \text{if }\chi_{\pi_{i}}=\lambda\\
0 & \text{else}
\end{cases}.\label{eq:averaging over a caharcter}
\end{equation}
Let $L:=\left\{ v\in V:\pi(h)v=\lambda(h)\cdot v,\,\,\forall h\in H\right\} $
be the subspace of $(H,\lambda)$-equivariant vectors in $V$ and
let $L^{\bot}$ be an $H$-invariant subspace of $V$ with $V=L\oplus L^{\bot}$.
By (\ref{eq:averaging over a caharcter}), the map $A:=\sum_{h\in H}\lambda^{-1}(h)\pi(h)\in\mathrm{End}(V)$
satisfies $A|_{L^{\bot}}=0$ and $A|_{L}=\left|H\right|\cdot I_{L}$.
Take an orthonormal basis $v_{1},\ldots,v_{N}$ for $V$ with $L=\langle v_{1},\ldots,v_{M}\rangle$,
$L^{\bot}=\langle v_{M+1},\ldots,v_{N}\rangle$. Then:
\[
\left|\sum_{h\in H}\lambda^{-1}(h)\chi_{\pi}(gh)\right|=\left|\sum_{i=1}^{N}\langle\pi(g)\left(\sum_{h\in H}\lambda^{-1}(h)\pi(h)\right)v_{i},v_{i}\rangle\right|=\left|H\right|\left|\sum_{i=1}^{M}\langle\pi(g)v_{i},v_{i}\rangle\right|\leq M\left|H\right|,
\]
and the lemma follows. 
\end{proof}
The following lemma gives a different estimate on the average of a
character over a coset, and this estimate is sharper when the double
coset $HgH$ is large. We will not need these alternative estimates,
but we thought it could be useful to state them. 
\begin{lem}
\label{lem:bounds on double cosets}Let $G$ be a finite group, and
let $H\leq G$ be a subgroup. Then, for each $\chi\in\mathrm{Irr}(G)$
and each $g\in G$, 
\[
\left|\frac{1}{\left|H\right|}\sum_{h\in H}\chi(hg)\right|\leq\frac{\langle\chi,1\rangle_{H}^{1/2}\cdot\left|G\right|^{1/2}}{\left|HgH\right|^{1/2}\chi(1)^{1/2}}.
\]
\end{lem}

\begin{proof}
Let $G$ be a finite group. For each $\chi\in\mathrm{Irr}(G)$, we
denote by $(\pi_{\chi},V_{\chi})$ the representation corresponding
to $\chi$. The non-commutative Fourier transform (see e.g.~\cite[Section 2.3]{App14})
is the map $\mathcal{F}:\C[G]\rightarrow\bigoplus_{\chi\in\mathrm{Irr}(G)}\mathrm{End}(V_{\chi})$
defined by $f\mapsto\widehat{f}:=\left(\widehat{f}(\chi)\right)_{\chi\in\mathrm{Irr}(G)}$,
where $\widehat{f}(\chi)=\frac{1}{\left|G\right|}\sum_{g'\in G}f(g')\pi_{\chi}(g'^{-1})$.
We denote by $\left\Vert f\right\Vert _{2}:=\left(\frac{1}{\left|G\right|}\sum_{g'\in G}\left|f(g')\right|^{2}\right)^{\frac{1}{2}}$.
Similarly, for a collection of endomorphisms $(A_{\chi})_{\chi\in\mathrm{Irr}(G)}\in\bigoplus_{\chi\in\mathrm{Irr}(G)}\mathrm{End}(V_{\chi})$,
with $A_{\chi}\in\mathrm{End}(V_{\chi})$, we define
\[
\left\Vert (A_{\chi})_{\chi\in\mathrm{Irr}(G)}\right\Vert _{2}:=\left(\sum_{\chi\in\mathrm{Irr}(G)}\chi(1)\cdot\left\Vert A_{\chi}\right\Vert _{\mathrm{HS}}^{2}\right)^{\frac{1}{2}},
\]
where $\left\Vert A_{\chi}\right\Vert _{\mathrm{HS}}:=\tr(A_{\chi}\cdot A_{\chi}^{*})^{\frac{1}{2}}$
is the Hilbert\textendash Schmidt norm on $\mathrm{End}(V_{\chi})$.
The Plancherel Theorem (see e.g.~\cite[Theorem 2.3.1(2)]{App14}),
states that
\begin{equation}
\left\Vert f\right\Vert _{2}=\left\Vert \widehat{f}\right\Vert _{2}.\label{eq:Plancherel}
\end{equation}
Let $\psi_{HgH}:=\frac{1}{\left|HgH\right|}1_{HgH}$. For each $\chi\in\mathrm{Irr}(G)$,
one has 
\[
\widehat{\psi_{HgH}}(\chi)=\frac{1}{\left|G\right|}\sum_{g'\in G}\psi_{HgH}(g')\pi_{\chi}(g'^{-1})=\frac{1}{\left|HgH\right|\left|G\right|}\sum_{g'\in HgH}\pi_{\chi}(g'^{-1}).
\]
The square of the $L^{2}$-norm of $\psi_{HgH}$ is given by: 
\begin{equation}
\left\Vert \psi_{HgH}\right\Vert _{2}^{2}=\frac{1}{\left|G\right|}\sum_{g'\in G}\left(\psi_{HgH}(g')\right)^{2}=\frac{1}{\left|G\right|}\sum_{g'\in HgH}\frac{1}{\left|HgH\right|^{2}}=\frac{1}{\left|HgH\right|\left|G\right|}.\label{eq:norm of function}
\end{equation}
Let $v_{1},\ldots,v_{M}$ be an orthonormal basis of $V_{\chi}^{H}:=\left\{ v\in V_{\chi}:\pi_{\chi}(h).v=v,\,\,\forall h\in H\right\} $
with respect to some $G$-invariant inner product $\langle\,,\,\rangle$
on $V_{\chi}$, with $M=\langle\chi,1\rangle_{H}$. Let $\left(V_{\chi}^{H}\right)^{\perp}$
be the orthogonal complement to $V_{\chi}^{H}$ in $V_{\chi}$. In
the proof of Lemma \ref{lem:primitive character bounds}, in the case
that $\lambda=1$, we have seen that 
\begin{equation}
\sum_{h\in H}\pi_{\chi}(h).v=\begin{cases}
0 & \text{if }v\in\left(V_{\chi}^{H}\right)^{\perp}\\
\left|H\right|\cdot v & \text{if }v\in V_{\chi}^{H}.
\end{cases}\label{eq:averaging operator}
\end{equation}
In particular, we have 
\begin{align*}
\left\langle \sum_{g'\in HgH}\pi_{\chi}(g'^{-1}).v,v\right\rangle  & =\frac{\left|HgH\right|}{\left|H\right|^{2}}\left\langle \sum_{h',h\in H}\pi_{\chi}(h'g^{-1}h).v,v\right\rangle =\frac{\left|HgH\right|}{\left|H\right|^{2}}\left\langle (\sum_{h'\in H}\pi_{\chi}(h')).(\sum_{h\in H}\pi_{\chi}(g^{-1}h).v),v\right\rangle \\
 & =\frac{\left|HgH\right|}{\left|H\right|^{2}}\left\langle \sum_{h\in H}\pi_{\chi}(g^{-1}h).v,\sum_{h'\in H}\pi_{\chi}(h')v\right\rangle =\begin{cases}
0 & \text{if }v\in\left(V_{\chi}^{H}\right)^{\perp}\\
\left|HgH\right|\langle\pi_{\chi}(g^{-1}).v,v\rangle & \text{if }v\in V_{\chi}^{H}.
\end{cases}
\end{align*}
Hence,
\begin{equation}
\left\Vert \widehat{\psi_{HgH}}\right\Vert _{2}^{2}=\sum_{\chi\in\mathrm{Irr}(G)}\chi(1)\left\Vert \frac{1}{\left|HgH\right|\left|G\right|}\sum_{g'\in HgH}\pi_{\chi}(g'^{-1})\right\Vert _{\mathrm{HS}}^{2}=\sum_{\chi\in\mathrm{Irr}(G)}\frac{\chi(1)}{\left|G\right|^{2}}\sum_{i,j=1}^{M}\left|\langle\pi_{\chi}(g^{-1}).v_{i},v_{j}\rangle\right|^{2},\label{eq:norm of Fourier}
\end{equation}
By (\ref{eq:Plancherel}), (\ref{eq:norm of function}) is equal to
(\ref{eq:norm of Fourier}), hence,
\begin{align*}
\left|\frac{1}{\left|H\right|}\sum_{h\in H}\chi(hg^{-1})\right|^{2} & =\left|\sum_{i=1}^{M}\langle\pi_{\chi}(g^{-1}).v_{i},v_{i}\rangle\right|^{2}\leq M\sum_{i=1}^{M}\left|\langle\pi_{\chi}(g^{-1}).v_{i},v_{i}\rangle\right|^{2}\\
 & \leq M\sum_{i,j=1}^{M}\left|\langle\pi_{\chi}(g^{-1}).v_{i},v_{j}\rangle\right|^{2}\leq\frac{M\left|G\right|}{\chi(1)\left|HgH\right|},
\end{align*}
where the first equality follows from (\ref{eq:averaging operator}),
and the first inequality follows from Cauchy\textendash Schwarz inequality.
\end{proof}

\subsection{\label{subsec:Weingarten-calculus}Weingarten calculus}

In $\mathsection\ref{subsec:Representation-theory-of symmetric}$,
\ref{subsec:Representation-theory-of unitary } we stated that each
partition $\lambda\vdash m$ with $\ell(\lambda)\leq d$ induces two
different representations, $\rho_{\lambda}\in\mathrm{Irr}(\U_{d})$
and $\chi_{\lambda}\in\mathrm{Irr}(S_{m})$. There is a deeper connection
between $\rho_{\lambda}$ and $\chi_{\lambda}$ coming from the\emph{
}Schur\textendash Weyl duality: the space $\left(\C^{d}\right)^{\otimes m}$
carries a natural action of $\U_{d}\times S_{m}$, where $A\in\U_{d}$
acts diagonally $A.\left(v_{1}\otimes\cdots\otimes v_{m}\right)=Av_{1}\otimes\cdots\otimes Av_{m}$,
and $\sigma\in S_{m}$ acts by $\sigma.\left(v_{1}\otimes\cdots\otimes v_{m}\right)=v_{\sigma(1)}\otimes\cdots\otimes v_{\sigma(m)}$.
The Schur\textendash Weyl duality can be phrased as follows.
\begin{thm}[Schur\textendash Weyl duality, \cite{Wey39}]
\label{thm:Schur-Weyl duality}The space $\left(\C^{d}\right)^{\otimes m}$
is a multiplicity-free representation of $\U_{d}\times S_{m}$. The
decomposition of $\left(\C^{d}\right)^{\otimes m}$ into irreducible
components is given by 
\begin{equation}
\left(\C^{d}\right)^{\otimes m}=\bigoplus_{\lambda\vdash m,\ell(\lambda)\leq d}\rho_{\lambda}\otimes\chi_{\lambda}.\label{eq:Schur-Weyl duality}
\end{equation}
\end{thm}

There are two special functions on $S_{m}$ which come from (\ref{eq:Schur-Weyl duality}).
Firstly, writing $\ell(\sigma)$ for the number of disjoint cycles
in $\sigma\in S_{m}$, the character of $\left(\C^{d}\right)^{\otimes m}$
as a representation of $S_{m}$ is the function $\sigma\mapsto d^{\ell(\sigma)}$. 

Recall we have an isomorphism of algebras $\C[S_{m}]\simeq\bigoplus_{\lambda\vdash m}\mathrm{End}(V_{\chi_{\lambda}})$,
where the multiplication in $\C[S_{m}]$ is the convolution operation
$f_{1}*f_{2}(y):=\sum_{x\in S_{m}}f(x)g(x^{-1}y)$. We denote by $\C_{d}[S_{m}]$
the subalgebra corresponding to $\bigoplus_{\lambda\vdash m,\ell(\lambda)\leq d}\mathrm{End}(V_{\chi_{\lambda}})$.
\begin{defn}[{\cite[Proposition 2.3]{CS06}}]
\label{def:Weingarten function}Let $d\in\N$. The \emph{Weingarten
function} $\Wg_{d}:S_{m}\rightarrow\C$ is the inverse of the function
$d^{\ell(\sigma)}$ in the ring $\C_{d}[S_{m}]$. It has the following
Fourier expansion: 
\begin{equation}
\Wg_{d}(\sigma)=\frac{1}{m!^{2}}\sum_{\lambda\vdash m,\ell(\lambda)\leq d}\frac{\chi_{\lambda}(1)^{2}}{\rho_{\lambda}(1)}\chi_{\lambda}(\sigma).\label{eq:Weingarten function}
\end{equation}
\end{defn}

\begin{rem}
Since in this paper we only consider $\Wg_{d'}(\sigma)$ for $d'=d$,
we write $\Wg$ instead of $\Wg_{d}$. 
\end{rem}

The \emph{Weingarten Calculus}, developed in \cite{Wein78,Col03,CS06}
utilizes the Schur\textendash Weyl duality to express integrals on
unitary groups as finite sums of Weingarten functions on symmetric
groups. One formulation is the following theorem by Collins and \'{S}niady:
\begin{thm}[{\cite[Corollary 2.4]{CS06}}]
\label{thm:Weingarten calculus}Let $(i_{1},\ldots,i_{m})$, $(j_{1},\ldots,j_{m})$,
$(i'_{1},\ldots,i'_{m})$, and $(j'_{1},\ldots,j'_{m})$ be tuples
of integers in $[d]$. Then:
\begin{align}
 & \mathbb{E}_{X\in\U_{d}}\left(X_{i_{1},j_{1}}\cdots X_{i_{m},j_{m}}\cdot\overline{X_{i'_{1},j'_{1}}\cdots X_{i'_{m},j'_{m}}}\right)\nonumber \\
= & \sum_{\sigma,\tau\in S_{m}}\delta_{i_{1},i'_{\sigma(1)}}\cdots\delta_{i_{m},i'_{\sigma(m)}}\delta_{j_{1},j'_{\tau(1)}}\cdots\delta_{j_{m},j'_{\tau(m)}}\cdot\Wg_{d}(\sigma^{-1}\tau).\label{eq:Weingarten calculus}
\end{align}
\end{thm}

We will use a coordinate-free version of Theorem \ref{thm:Weingarten calculus}
which we proceed to state.
\begin{defn}
\label{def:symmetric partition and generalizaed Weingarten}Let $\Omega$
be a set. 
\begin{enumerate}
\item A \emph{symmetric partition} $\Phi$ of $\Omega$ is a partition $\Omega=\bigsqcup_{i=1}^{r}A_{i}\sqcup\bigsqcup_{i=1}^{r}B_{i}$,
where $\left|A_{i}\right|=\left|B_{i}\right|$. 
\item Given a symmetric partition $\Phi=(A_{1},\ldots,A_{r},B_{1},\ldots,B_{r})$,
let 
\[
S_{\Phi}=\left\{ \varSigma\in\Sym(\Omega):\varSigma(A_{i})=B_{i},\varSigma(B_{i})=A_{i}\right\} .
\]
\item If $\varSigma\in S_{\Phi}$, then $\varSigma^{2}(A_{i})=A_{i}$ and
we define $\widetilde{\Wg}(\varSigma^{2})=\prod_{i=1}^{r}\Wg(\varSigma^{2}|_{A_{i}})$.
\end{enumerate}
\end{defn}

\begin{prop}
\label{prop:Weingarten coordinate free}Let $\Phi=(A,B)$ be a symmetric
partition of $\Omega$ and let $F,H:\Omega\rightarrow[d]$. Then 
\[
\mathbb{E}_{X\in\U_{d}}\left(\prod_{x\in A}X_{F(x),H(x)}\prod_{y\in B}X_{F(y),H(y)}^{-1}\right)=\mathbb{E}\left(\prod_{x\in A}X_{F(x),H(x)}\prod_{y\in B}\overline{X_{H(y),F(y)}}\right)=\sum_{\Sigma\in S_{\Phi}:\,H=F\circ\Sigma}\widetilde{\Wg}(\Sigma^{2}).
\]
\end{prop}

\begin{proof}
Identify $A\cong\left\{ 1,\ldots,m\right\} $ and $B\cong\left\{ -1,\ldots,-m\right\} $
and let $\overrightarrow{i},\overrightarrow{j},\overrightarrow{i}',\overrightarrow{j}'\in[d]^{m}$
be
\[
i_{k}=F(k)\quad j_{k}=H(k)\quad i'_{k}=H(-k)\quad j'_{k}=F(-k).
\]
Then, by Theorem \ref{thm:Weingarten calculus},
\begin{align*}
\mathbb{E}_{X\in\mathrm{\U}_{d}}\left(\prod_{x\in A}X_{F(x),H(x)}\prod_{y\in B}X_{F(y),H(y)}^{-1}\right) & =\mathbb{E}_{X\in\mathrm{\U}_{d}}\left(X_{i_{1},j_{1}}\cdots X_{i_{m},j_{m}}\overline{X_{i'_{1},j'_{1}}\cdots X_{i'_{m},j'_{m}}}\right)\\
 & =\sum_{\sigma,\tau\in S_{m}}\delta_{i_{1},i'_{\sigma(1)}}\cdots\delta_{i_{m},i'_{\sigma(m)}}\cdot\delta_{j_{1},j'_{\tau(1)}}\cdots\delta_{j_{m},j'_{\tau(m)}}\cdot\Wg(\sigma^{-1}\tau).
\end{align*}
For $\sigma,\tau\in S_{m}$, let $\Sigma_{(\sigma,\tau)}\in\Sym(A\sqcup B)\cong\Sym\left(\left\{ -m,\ldots,-1,1,\ldots,m\right\} \right)$
be the permutation 
\[
\Sigma_{(\sigma,\tau)}(x)=\begin{cases}
-\tau(x) & x\in\left\{ 1,\ldots,m\right\} \\
\sigma^{-1}(-x) & x\in\left\{ -1,\ldots,-m\right\} .
\end{cases}
\]
The map $(\sigma,\tau)\mapsto\Sigma_{(\sigma,\tau)}$ is a bijection
$S_{m}^{2}\cong S_{\Phi}$ and the condition $\delta_{i_{1},i'_{\sigma(1)}}\cdots\delta_{i_{m},i'_{\sigma(m)}}\cdot\delta_{j_{1},j'_{\tau(1)}}\cdots\delta_{j_{m},j'_{\tau(m)}}=1$
is equivalent to $H=F\circ\varSigma_{(\sigma,\tau)}$. Finally, the
permutation $\left(\varSigma_{(\sigma,\tau)}\right)^{2}$ acts on
$A$ as $\sigma^{-1}\tau$, and the result follows.
\end{proof}
\begin{cor}
\label{cor:Weingarten coordinate free- multidimensional}Let $\Phi=(A_{1},\ldots,A_{r},B_{1},\ldots,B_{r})$
be a symmetric partition of $\Omega$ and let $F,H:\Omega\rightarrow[d]$.
Then 
\[
\mathbb{E}\left(\prod_{i=1}^{r}\left(\prod_{x\in A_{i}}(X_{i})_{F(x),H(x)}\prod_{y\in B_{i}}(X_{i}^{-1})_{F(y),H(y)}\right)\right)=\sum_{\varSigma\in S_{\Phi}:\,H=F\circ\Sigma}\widetilde{\Wg}(\varSigma^{2}).
\]
\end{cor}

\section{\label{sec:The-Engel-word}The Engel word as a model case}
\begin{quotation}
\emph{``Those who run to long words are mainly the unskillful and
tasteless; they confuse pomposity with dignity, flaccidity with ease,
and bulk with force.''}\footnote{H.W Fowler, \emph{A Dictionary of Modern English Usage}, 1965.} 
\end{quotation}
In this section we prove the following simplified version of Theorem
\ref{thm:abs_value_of_trace} for the Engel word. We chose the Engel
word since it is short enough to make the proof easier to digest,
while at same time complicated enough so that the proof contains most
of the key ideas in the paper.
\begin{thm}
\label{thm:Main theorem for Engel word}Let $X,Y$ be independent
random variables with respect to the normalized Haar measure on $\U_{d}$.
For every $d\geq2m$, one has:
\[
\mathbb{E}\left(c_{m}\left([[X,Y],Y]\right)\right)<2^{17m}.
\]
\end{thm}

Let $w=[[x,y],y]=xyx^{-1}yxy^{-1}x^{-1}y^{-1}$ be the Engel word.
We would like to compute $\mathbb{E}\left(\tr\bigwedge\nolimits ^{m}w(X,Y)\right)$.
Denote $\mathcal{I}_{m,d}:=\{a_{1}<\ldots<a_{m}:a_{i}\in[d]\}$, and
note that
\begin{equation}
\tr\left(\bigwedge\nolimits ^{m}w(X,Y)\right)=\sum_{\overrightarrow{a}\in\mathcal{I}_{m,d}}\sum_{\pi\in S_{m}}(-1)^{\pi}w(X,Y)_{a_{1}a_{\pi(1)}}\cdots w(X,Y)_{a_{m}a_{\pi(m)}}.\label{eq:formula for trace of wedge}
\end{equation}
We have 
\begin{align}
w(X,Y)_{a_{i}a_{\pi(i)}} & =\sum_{b_{i},c_{i},d_{i},A_{i},B_{i},C_{i},D_{i}\in[d]}X_{a_{i},D_{i}}Y_{D_{i},c_{i}}X_{c_{i},A_{i}}^{-1}Y_{A_{i},b_{i}}X_{b_{i},C_{i}}Y_{C_{i},d_{i}}^{-1}X_{d_{i},B_{i}}^{-1}Y_{B_{i},a_{\pi(i)}}^{-1}\nonumber \\
 & =\sum_{b_{i},c_{i},d_{i},A_{i},B_{i},C_{i},D_{i}\in[d]}X_{a_{i},D_{i}}X_{b_{i},C_{i}}\overline{X_{A_{i},c_{i}}X_{B_{i},d_{i}}}Y_{A_{i},b_{i}}Y_{D_{i},c_{i}}\overline{Y_{a_{\pi(i)},B_{i}}Y_{d_{i},C_{i}}}.\label{eq:formula for w_ij}
\end{align}
The group $S_{m}$ acts on $[d]^{m}$ by $\sigma(\overrightarrow{v})_{i}=\overrightarrow{v}_{\sigma^{-1}(i)}$
for any $\sigma\in S_{m}$ and $\overrightarrow{v}\in[d]^{m}$. Similarly,
given $\overrightarrow{v},\overrightarrow{w}\in[d]^{m}$ and $\tau\in S_{2m}$,
we denote by $(\overrightarrow{v},\overrightarrow{w})$ the element
in $[d]^{2m}$ given by $(\overrightarrow{v},\overrightarrow{w})_{i}=\begin{cases}
\overrightarrow{v}_{i} & \text{if }i\leq m\\
\overrightarrow{w}_{i-m} & \text{if }m<i\leq2m
\end{cases}$, and denote by $\tau(\overrightarrow{v},\overrightarrow{w})_{i}=(\overrightarrow{v},\overrightarrow{w})_{\tau^{-1}(i)}$.
In particular, writing $X_{\overrightarrow{v},\overrightarrow{u}}:=\prod_{i=1}^{m}X_{v_{i},u_{i}}$
for $\overrightarrow{v},\overrightarrow{u}\in[d]^{m}$, we have:
\begin{equation}
\tr\left(\bigwedge\nolimits ^{m}w(X,Y)\right)=\sum_{\overrightarrow{a}\in\mathcal{I}_{m,d}}\sum_{\overrightarrow{b},\ldots,\overrightarrow{D}\in[d]^{m}}\sum_{\pi\in S_{m}}(-1)^{\pi}\left(X_{\overrightarrow{a},\overrightarrow{D}}X_{\overrightarrow{b},\overrightarrow{C}}\overline{X_{\overrightarrow{A},\overrightarrow{c}}}\overline{X_{\overrightarrow{B},\overrightarrow{d}}}\right)\left(Y_{\overrightarrow{A},\overrightarrow{b}}Y_{\overrightarrow{D},\overrightarrow{c}}\overline{Y_{\pi^{-1}(\overrightarrow{a}),\overrightarrow{B}}Y_{\overrightarrow{d},\overrightarrow{C}}}\right).\label{eq:refined formula for trace}
\end{equation}
We now rewrite the expected value of (\ref{eq:refined formula for trace})
using Weingarten calculus. For this, define: 
\[
S(\overrightarrow{a},\ldots,\overrightarrow{D}):=\left\{ (\sigma_{1},\sigma_{2},\tau_{1},\tau_{2})\in S_{2m}^{4}:\begin{array}{c}
(\overrightarrow{A},\overrightarrow{B})=\sigma_{1}(\overrightarrow{a},\overrightarrow{b}),\,\,\,(\overrightarrow{c},\overrightarrow{d})=\tau_{1}(\overrightarrow{D},\overrightarrow{C})\\
(\overrightarrow{a},\overrightarrow{d})=\sigma_{2}(\overrightarrow{A},\overrightarrow{D}),\,\,\,(\overrightarrow{B},\overrightarrow{C})=\tau_{2}(\overrightarrow{b},\overrightarrow{c})
\end{array}\right\} ,
\]
and
\begin{equation}
Z:=\left\{ (\overrightarrow{a},\ldots,\overrightarrow{D},\sigma_{1},\sigma_{2},\tau_{1},\tau_{2})\in\mathcal{I}_{m,d}\times[d]^{7m}\times S_{2m}^{4}:(\sigma_{1},\sigma_{2},\tau_{1},\tau_{2})\in S(\overrightarrow{a},\ldots,\overrightarrow{D})\right\} .\label{eq:def of Z}
\end{equation}

\begin{lem}
\label{lem:reduction using Weingarten}We have:
\begin{equation}
\mathbb{E}\left(\tr\bigwedge\nolimits ^{m}w(X,Y)\right)=\sum_{(\overrightarrow{a},\ldots,\overrightarrow{D},\sigma_{1},\sigma_{2},\tau_{1},\tau_{2})\in Z}\sum_{\pi\in S_{m}}(-1)^{\pi}\Wg(\sigma_{1}^{-1}\tau_{1})\Wg(\sigma_{2}^{-1}\mathrm{(\pi}\times\mathrm{Id})\tau_{2}).\label{eq:reduction using Weingarten}
\end{equation}
\end{lem}

\begin{proof}
Using Weingarten calculus, i.e., Theorem \ref{thm:Weingarten calculus},
and (\ref{eq:refined formula for trace}), 
\begin{align}
\mathbb{E}\left(\tr\bigwedge\nolimits ^{m}w(X,Y)\right) & =\sum_{\overrightarrow{a}\in\mathcal{I}_{m,d}}\sum_{\overrightarrow{b},\ldots,\overrightarrow{D}\in[d]^{m}}\sum_{\pi\in S_{m}}(-1)^{\pi}\sum_{\sigma_{1},\widetilde{\sigma}_{2},\tau_{1},\tau_{2}\in S_{2m}}\delta_{(\overrightarrow{a},\overrightarrow{b}),\sigma_{1}^{-1}(\overrightarrow{A},\overrightarrow{B})}\cdot\delta_{(\overrightarrow{D},\overrightarrow{C}),\tau_{1}^{-1}(\overrightarrow{c},\overrightarrow{d})}\Wg(\sigma_{1}^{-1}\tau_{1})\nonumber \\
 & \cdot\delta_{(\overrightarrow{A},\overrightarrow{D}),\widetilde{\sigma}_{2}^{-1}(\pi^{-1}(\overrightarrow{a}),\overrightarrow{d})}\cdot\delta_{(\overrightarrow{b},\overrightarrow{c}),\tau_{2}^{-1}(\overrightarrow{B},\overrightarrow{C})}\Wg(\widetilde{\sigma}_{2}^{-1}\tau_{2}).\label{eq:break the formula}
\end{align}
Applying the change of coordinate $\sigma_{2}:=\mathrm{(\pi}\times\mathrm{Id})\circ\widetilde{\sigma}_{2}$,
and observing that $\widetilde{\sigma}_{2}^{-1}(\pi^{-1}(\overrightarrow{a}),\overrightarrow{d})=\sigma_{2}^{-1}(\overrightarrow{a},\overrightarrow{d})$,
(\ref{eq:break the formula}) becomes:
\[
\mathbb{E}\left(\tr\bigwedge\nolimits ^{m}w(X,Y)\right)=\sum_{(\overrightarrow{a},\ldots,\overrightarrow{D},\sigma_{1},\sigma_{2},\tau_{1},\tau_{2})\in Z}\sum_{\pi\in S_{m}}(-1)^{\pi}\Wg(\sigma_{1}^{-1}\tau_{1})\cdot\Wg(\sigma_{2}^{-1}\mathrm{(\pi}\times\mathrm{Id})\tau_{2}).\qedhere
\]
\end{proof}
In order to bound (\ref{eq:reduction using Weingarten}), we consider
a natural action of $S_{m}^{7}$ on $Z$, and find a suitable change
of coordinates such that the average of the product of the Weingarten
functions in (\ref{eq:reduction using Weingarten}) over any $S_{m}^{7}$-orbit
is equal to a product of averages of individual Weingarten functions
over cosets (see (\ref{eq:summing over an L-orbit is nice})). We
then use Lemma \ref{lem:primitive character bounds} to estimate the
contribution in (\ref{eq:reduction using Weingarten}) of each $S_{m}^{7}$-orbit.
To conclude the estimates of (\ref{eq:reduction using Weingarten}),
we will further provide estimates for $\left|Z\right|$. 

We first describe the action of $S_{m}^{7}$. The element $(\pi_{b},\pi_{c},\ldots,\pi_{D})\in S_{m}^{7}$
acts on $(\overrightarrow{a},\ldots,\overrightarrow{D})$ by $(\overrightarrow{a},\pi_{b}(\overrightarrow{b}),\pi_{c}(\overrightarrow{c}),\ldots,\pi_{D}(\overrightarrow{D}))$
and it acts on $(\sigma_{1},\sigma_{2},\tau_{1},\tau_{2})$ by: 
\[
\sigma_{1}\mapsto(\pi_{A}\times\pi_{B})\circ\sigma_{1}\circ(\mathrm{Id}\times\pi_{b}^{-1})
\]
\[
\tau_{1}\mapsto(\pi_{c}\times\pi_{d})\circ\tau_{1}\circ(\pi_{D}^{-1}\times\pi_{C}^{-1})
\]
\[
\sigma_{2}\mapsto(\mathrm{Id}\times\pi_{d})\circ\sigma_{2}\circ(\pi_{A}^{-1}\times\pi_{D}^{-1})
\]
\[
\tau_{2}\mapsto(\pi_{B}\times\pi_{C})\circ\tau_{2}\circ(\pi_{b}^{-1}\times\pi_{c}^{-1}).
\]
This gives rise to an action of $S_{m}^{7}$ on $Z$. The action on
the input of the Weingarten functions becomes
\begin{equation}
\Wg((\pi_{D}^{-1}\times\pi_{C}^{-1}\pi_{b})\sigma_{1}^{-1}(\pi_{A}^{-1}\pi_{c}\times\pi_{B}^{-1}\pi_{d})\tau_{1})\text{ and }\Wg(\pi_{b}^{-1}\pi_{A}\times\pi_{c}^{-1}\pi_{D})\sigma_{2}^{-1}(\pi\pi_{B}\times\pi_{d}^{-1}\pi_{C})\tau_{2}),\label{eq:Wg_1}
\end{equation}
where we used the conjugacy invariance of $\Wg$ to move permutations
from right to left. Consider the bijection $\psi:S_{m}^{8}\rightarrow S_{m}^{8}$,
defined by $(x_{1},\ldots,x_{8})\mapsto(x_{1},x_{1}x_{2},\ldots,x_{1}x_{2}\cdots x_{8})$.
Under the change of coordinates $(\theta_{D},\theta_{c},\theta_{A},\theta_{b},\theta_{C},\theta_{d},\theta_{B},\theta):=\psi^{-1}(\pi_{D},\pi_{c},\pi_{A},\pi_{b},\pi_{C},\pi_{d},\pi_{B},\pi^{-1})$,
the summation of (\ref{eq:reduction using Weingarten}) over an $S_{m}^{7}$-orbit
splits into a product of two separate sums: 
\begin{align}
 & \sum_{(\pi_{D},\ldots,\pi)\in S_{m}^{8}}(-1)^{\pi}\Wg((\pi_{D}^{-1}\times\pi_{C}^{-1}\pi_{b})\sigma_{1}^{-1}(\pi_{A}^{-1}\pi_{c}\times\pi_{B}^{-1}\pi_{d})\tau_{1})\Wg(\pi_{b}^{-1}\pi_{A}\times\pi_{c}^{-1}\pi_{D})\sigma_{2}^{-1}(\mathrm{\pi}\pi_{B}\times\pi_{d}^{-1}\pi_{C})\tau_{2})\nonumber \\
= & \sum_{(\theta_{D},\ldots,\theta)\in S_{m}^{8}}(-1)^{\theta_{D}\cdots\theta}\Wg((\theta_{D}^{-1}\times\theta_{C}^{-1})\sigma_{1}^{-1}(\theta_{A}^{-1}\times\theta_{B}^{-1})\tau_{1})\Wg(\theta_{b}^{-1}\times\theta_{c}^{-1})\sigma_{2}^{-1}(\mathrm{\theta}^{-1}\times\theta_{d}^{-1})\tau_{2})\nonumber \\
= & \sum_{\eta_{1},\eta'_{1}\in S_{m}^{2}}(-1)^{\eta_{1}\eta'_{1}}\Wg(\eta_{1}\sigma_{1}^{-1}\eta'_{1}\tau_{1})\sum_{\eta_{2},\eta'_{2}\in S_{m}^{2}}(-1)^{\eta_{2}\eta'_{2}}\Wg(\eta_{2}\sigma_{2}^{-1}\eta'_{2}\tau_{2}).\label{eq:summing over an L-orbit is nice}
\end{align}
We can now use the Fourier expansion of $\Wg$ (\ref{eq:Weingarten function})
and the estimates in $\mathsection$\ref{subsec:Averaging-characters-over}
to bound the contribution of an $S_{m}^{7}$-orbit in $Z$ to (\ref{eq:reduction using Weingarten}):
\begin{prop}
\label{prop:estimates on a single L-orbit Engel}Let $\widetilde{v}:=(\widetilde{\overrightarrow{a}},\ldots,\widetilde{\overrightarrow{D}},\widetilde{\sigma}_{1},\widetilde{\sigma}_{2},\widetilde{\tau}_{1},\widetilde{\tau}_{2})\in Z$
and let $\mathcal{O}_{\widetilde{v}}:=S_{m}^{7}\widetilde{v}$ be
its $S_{m}^{7}$-orbit. Then, 
\begin{equation}
\left|\frac{1}{\left|\mathcal{O}_{\widetilde{v}}\right|}\sum_{(\overrightarrow{a},\ldots,\tau_{2})\in\mathcal{O}_{\widetilde{v}}}\sum_{\pi\in S_{m}}(-1)^{\pi}\Wg(\sigma_{1}^{-1}\tau_{1})\Wg(\sigma_{2}^{-1}\mathrm{(\pi}\times\mathrm{Id})\tau_{2})\right|\leq\frac{1}{(2m)!^{2}m!^{3}}\binom{d}{2m}^{-2}.\label{eq:Prop 3.3}
\end{equation}
\end{prop}

\begin{proof}
By the Orbit-Stabilizer Theorem, the LHS of (\ref{eq:Prop 3.3}) is
the same as summing over all $(\pi_{D},\ldots,\pi_{B})\in S_{m}^{7}$
and dividing by $m!^{7}$. By (\ref{eq:summing over an L-orbit is nice}),
the LHS of (\ref{eq:Prop 3.3}) is equal to 
\[
\frac{1}{m!^{7}}\left|\sum_{\eta_{1},\eta'_{1}\in S_{m}^{2}}(-1)^{\eta_{1}\eta'_{1}}\Wg(\eta_{1}\widetilde{\sigma}_{1}^{-1}\eta'_{1}\widetilde{\tau}_{1})\sum_{\eta_{2},\eta'_{2}\in S_{m}^{2}}(-1)^{\eta_{2}\eta'_{2}}\Wg(\eta_{2}\widetilde{\sigma}_{2}^{-1}\eta'_{2}\widetilde{\tau}_{2})\right|.
\]
Note that $(S_{2m},S_{m}\times S_{m})$ is a $\sgn$-twisted Gelfand
pair, that is, the representation $\mathrm{Ind}_{S_{m}^{2}}^{S_{2m}}\sgn$
is multiplicity-free. By Frobenius reciprocity, each irreducible subrepresentation
$(V_{\lambda},\pi_{\lambda})$ of $\mathrm{Ind}_{S_{m}^{2}}^{S_{2m}}\sgn$
has a unique $(S_{m}^{2},\sgn)$-invariant unit vector, so $\left\langle \chi_{\lambda}|_{S_{m}^{2}},\sgn\right\rangle _{S_{m}^{2}}=1$.
By Lemma \ref{lem:primitive character bounds}, for each $\sigma\in S_{2m}$,
we have
\begin{equation}
\left|\sum_{h\in S_{m}^{2}}(-1)^{h}\chi_{\lambda}(h\sigma)\right|\leq m!^{2}\left\langle \chi_{\lambda}|_{S_{m}^{2}},\sgn\right\rangle _{S_{m}^{2}}=\begin{cases}
m!^{2} & \text{if }\pi_{\lambda}\hookrightarrow\mathrm{Ind}_{S_{m}^{2}}^{S_{2m}}\sgn\\
0 & \text{otherwise.}
\end{cases}.\label{eq:averaging character over coset-Engel}
\end{equation}
By Lemma \ref{lem:description of littlewood richardson for Young subgroups}
it follows that $\pi_{\lambda}\hookrightarrow\mathrm{Ind}_{S_{m}^{2}}^{S_{2m}}\sgn$
if and only if the young diagram of $\lambda\vdash2m$ has at most
two columns. Combining with (\ref{eq:averaging character over coset-Engel}),
we have:
\begin{align}
 & \left|\sum_{\eta_{1},\eta'_{1}\in S_{m}^{2}}(-1)^{\eta_{1}\eta'_{1}}\chi_{\lambda}(\eta_{1}\widetilde{\sigma}_{1}^{-1}\eta'_{1}\widetilde{\tau}_{1})\right|=\left|\sum_{\eta_{1}'\in S_{m}^{2}}(-1)^{\eta'_{1}}\sum_{\eta_{1}\in S_{m}^{2}}(-1)^{\eta_{1}}\chi_{\lambda}(\eta_{1}\widetilde{\sigma}_{1}^{-1}\eta'_{1}\widetilde{\tau}_{1})\right|\label{eq:more averaging}\\
\leq & \sum_{\eta_{1}'\in S_{m}^{2}}\left|\sum_{\eta_{1}\in S_{m}^{2}}(-1)^{\eta_{1}}\chi_{\lambda}(\eta_{1}\widetilde{\sigma}_{1}^{-1}\eta'_{1}\widetilde{\tau}_{1})\right|\leq\begin{cases}
m!^{4} & \text{if }\lambda\vdash2m,\,\lambda\text{ has}\leq2\text{ columns},\\
0 & \text{otherwise.}
\end{cases}\nonumber 
\end{align}
By (\ref{eq:Weingarten function}), (\ref{eq:more averaging}), (\ref{eq:formula for dimension}),
and by our assumption that $d\geq2m$, we have
\begin{align*}
 & \left|\sum_{\eta_{1},\eta'_{1}\in S_{m}^{2}}(-1)^{\eta_{1}\eta'_{1}}\Wg(\eta_{1}\widetilde{\sigma}_{1}^{-1}\eta'_{1}\widetilde{\tau}_{1})\right|=\frac{1}{(2m)!^{2}}\left|\sum_{\lambda\vdash2m}\frac{\chi_{\lambda}(1)^{2}}{\rho_{\lambda}(1)}\sum_{\eta_{1},\eta'_{1}\in S_{m}^{2}}(-1)^{\eta_{1}\eta'_{1}}\chi_{\lambda}(\eta_{1}\widetilde{\sigma}_{1}^{-1}\eta'_{1}\widetilde{\tau}_{1})\right|\\
\leq & \frac{m!^{4}}{(2m)!^{2}}\sum_{\lambda\vdash2m,\,\lambda\text{ has }\leq2\text{ columns}}\frac{\chi_{\lambda}(1)^{2}}{\rho_{\lambda}(1)}=\frac{m!^{4}}{(2m)!}\sum_{\lambda\vdash2m,\,\lambda\text{ has }\leq2\text{ columns}}\frac{\chi_{\lambda}(1)}{\prod_{(i,j)\in\lambda}(d+j-i)}\\
\leq & \frac{m!^{4}}{(2m)!}\cdot\frac{1}{d\cdot\cdots\cdot(d-2m+1)}\sum_{\lambda\vdash2m,\,\lambda\text{ has }\leq2\text{ columns}}\chi_{\lambda}(1)=\frac{m!^{4}}{(2m)!}\frac{\dim\mathrm{Ind}_{S_{m}^{2}}^{S_{2m}}\sgn}{d\cdot\cdots\cdot(d-2m+1)}=\frac{m!^{2}}{(2m)!}\binom{d}{2m}^{-1}.
\end{align*}
This concludes the proposition. 
\end{proof}
We now turn to the last ingredient in the proof of Theorem \ref{thm:Main theorem for Engel word}. 
\begin{defn}
\label{def:shape of a function}Let $f:S\rightarrow[d]$ be a function
on a set $S$. We define the \emph{shape} $\nu_{f}:[d]\rightarrow\N$
of $f$ as 
\[
\nu_{f}=(\nu_{f,1},\ldots,\nu_{f,d}):=(\left|f^{-1}(1)\right|,\ldots,\left|f^{-1}(d)\right|),
\]
and denote $\nu_{f}!:=\prod_{u=1}^{d}\nu_{f,u}$.
\end{defn}

\begin{prop}
\label{prop:upper bound on Z Engel}Let $Z$ be as in (\ref{eq:def of Z}).
Then:
\[
\left|Z\right|\leq m!^{7}\binom{2m}{m}^{4}\binom{d}{m}\binom{d+m-1}{m}^{3}.
\]
\end{prop}

\begin{proof}
We need to count all the possible tuples $(\overrightarrow{a},\ldots,\overrightarrow{D},\sigma_{1},\sigma_{2},\tau_{1},\tau_{2})$
in $Z$. Suppose we have already fixed $\overrightarrow{a}$ and the
shapes $\nu_{\overrightarrow{b}},\text{\ensuremath{\nu_{\overrightarrow{c}}} and \ensuremath{\nu_{\overrightarrow{d}}}}$
of $\overrightarrow{b},\overrightarrow{c}$ and $\overrightarrow{d}$,
where $\overrightarrow{b},\overrightarrow{c},\overrightarrow{d}$
are considered as a functions $[m]\rightarrow[d]$. Given this data:
\begin{enumerate}
\item There are $\frac{m!^{3}}{\nu_{\overrightarrow{b}}!\nu_{\overrightarrow{c}}!\nu_{\overrightarrow{d}}!}$
options for $\overrightarrow{b},\overrightarrow{c},\overrightarrow{d}$
with the above shapes. 
\item There are $(2m)!^{2}$ options for $\sigma_{2}$ and $\tau_{2}$.
\item There are at most $\binom{2m}{m}^{2}$ options for choosing $\tau_{1}^{-1}([m])$
and $\sigma_{1}([m])$, as subsets of $[2m]$. Note that we count
both valid and invalid options.
\item After fixing the subsets $\tau_{1}^{-1}([m])$ and $\sigma_{1}([m])$,
there are at most $\nu_{\overrightarrow{c}}!\nu_{\overrightarrow{d}}!$
options for $\tau_{1}$ and $\nu_{\overrightarrow{b}}!$ options for
$\sigma_{1}$. 
\end{enumerate}
Summarizing the above items, we get there are at most $\frac{m!^{3}(2m)!^{2}\nu_{\overrightarrow{b}}!\nu_{\overrightarrow{c}}!\nu_{\overrightarrow{d}}!}{\nu_{\overrightarrow{b}}!\nu_{\overrightarrow{c}}!\nu_{\overrightarrow{d}}!}\binom{2m}{m}^{2}=m!^{7}\binom{2m}{m}^{4}$
options for $(\overrightarrow{a},\ldots,\overrightarrow{D},\sigma_{1},\sigma_{2},\tau_{1},\tau_{2})\in Z$
with the initial data $\overrightarrow{a},\nu_{\overrightarrow{b}},\text{\ensuremath{\nu_{\overrightarrow{c}}}},\ensuremath{\nu_{\overrightarrow{d}}}$.
Note that there are $\binom{d}{m}$ possible options for $\overrightarrow{a}$,
and $\binom{d+m-1}{m}^{3}$ options for $\nu_{\overrightarrow{b}},\text{\ensuremath{\nu_{\overrightarrow{c}}}},\ensuremath{\nu_{\overrightarrow{d}}}$.
This gives the desired upper bound.
\end{proof}
We can now finish the proof of Theorem \ref{thm:Main theorem for Engel word}. 
\begin{proof}[Proof of Theorem \ref{thm:Main theorem for Engel word}]
Note that for every $k\geq1$ and $n\geq k$ we have
\begin{equation}
\left(\frac{n}{k}\right)^{k}\leq\prod_{j=0}^{k-1}\left(\frac{n-j}{k-j}\right)=\binom{n}{k}\leq\frac{n^{k}}{k!}\leq\left(\frac{n}{k}\right)^{k}e^{k},\label{eq:Striling on binomial coefficients}
\end{equation}
where the rightmost inequality follows from Stirling's approximation.
By Lemma \ref{lem:reduction using Weingarten} and by Propositions
\ref{prop:estimates on a single L-orbit Engel} and \ref{prop:upper bound on Z Engel},
\begin{align}
\left|\mathbb{E}\left(\tr\bigwedge\nolimits ^{m}w(X,Y)\right)\right| & =\left|Z\right|\cdot\left|\frac{1}{\left|Z\right|}\sum_{(\overrightarrow{a},\ldots,\overrightarrow{D},\sigma_{1},\sigma_{2},\tau_{1},\tau_{2})\in Z}\sum_{\pi\in S_{m}}(-1)^{\pi}\Wg(\sigma_{1}^{-1}\tau_{1})\cdot\Wg(\sigma_{2}^{-1}\mathrm{(\pi}\times\mathrm{Id})\tau_{2})\right|\nonumber \\
 & \leq\left|Z\right|\cdot\frac{1}{(2m)!^{2}m!^{3}}\binom{d}{2m}^{-2}\leq\binom{2m}{m}^{2}\binom{d}{m}\binom{d+m}{m}^{3}\binom{d}{2m}^{-2}.\label{eq:auxilary inequality}
\end{align}
By (\ref{eq:auxilary inequality}), (\ref{eq:Striling on binomial coefficients}),
by the inequality $\binom{2m}{m}\leq2^{2m}$, and by our assumption
that $d\geq2m$, 
\[
\left|\mathbb{E}\left(\tr\bigwedge\nolimits ^{m}w(X,Y)\right)\right|\leq\frac{2^{4m}e^{4m}\left(\frac{d}{m}\right)^{m}\left(\frac{d+m}{m}\right)^{3m}}{\left(\frac{d}{2m}\right)^{4m}}\leq\frac{2^{7m}e^{4m}\left(\frac{d}{m}\right)^{4m}}{\left(\frac{d}{2m}\right)^{4m}}\leq2^{11m}e^{4m}\leq2^{17m}.\qedhere
\]
\end{proof}
\begin{rem}
\label{rem:generalization to arbitrary words} The current proof of
Proposition \ref{prop:upper bound on Z Engel} depends on the special
structure of the Engel word. One can give a slightly more complicated
argument, which can be easily generalized for every word $w$ (this
is done in $\mathsection$\ref{sec:Estimates-on}). Here are the main
ideas of this alternative argument.

We encode the expression 
\begin{equation}
X_{a,D}Y_{D,c}X_{c,A}^{-1}Y_{A,b}X_{b,C}Y_{C,d}^{-1}X_{d,B}^{-1}Y_{B,a}^{-1}\label{eq:encoding the Engel}
\end{equation}
from (\ref{eq:formula for w_ij}), graphically, by the $4\times4$
matrix
\begin{equation}
\left(\begin{array}{cccc}
\cdot & C & D & \cdot\\
c & \cdot & \cdot & b\\
d & \cdot & \cdot & a\\
\cdot & B & A & \cdot
\end{array}\right),\label{eq:matrix}
\end{equation}
which is constructed as follows. The rows and columns are indexed
by $x,y,x^{-1},y^{-1}$. We order the rows by $x<y<y^{-1}<x^{-1}$
and order the columns by $x^{-1}<y^{-1}<y<x$. To find the $(x,y^{-1})$-entry
of this matrix (i.e.~the $(1,2)$-entry), we look for the subword
$XY^{-1}$ in (\ref{eq:encoding the Engel}) and record the letter
of the common index, which is $C$. All other entries are determined
in similar fashion. Note that we do not have elements in the main
diagonal since $w$ is cyclically reduced. 

We denote $\eta_{1}=\tau_{1}$, $\eta_{2}=\tau_{2}$, $\eta_{3}=\sigma_{2}^{-1}$,
$\eta_{4}=\sigma_{1}^{-1}$. Note that $\eta_{i}$ sends the $i$th
row of (\ref{eq:matrix}) into a permuted copy of its $i$th column.
The alternative counting argument in Proposition \ref{prop:upper bound on Z Engel}
goes as follows. We fix the upper triangular part, i.e., $\overrightarrow{C},\overrightarrow{D},\overrightarrow{a},\overrightarrow{b}$
(instead of $\overrightarrow{a},\overrightarrow{b},\overrightarrow{c},\overrightarrow{d}$
in the proof above). We then choose $\eta_{1}$ (with $2m!$ options),
which gives us $\overrightarrow{c},\overrightarrow{d}$ and, in particular,
reveals the second row. Next, we choose all possible $\eta_{2}:(\overrightarrow{b},\overrightarrow{c})\rightarrow(\overrightarrow{B},\overrightarrow{C})$,
taking into consideration the fact that $\overrightarrow{C}$ is already
known. We then proceed to the next row and guess $\eta_{3}$, taking
into consideration that we already know $\overrightarrow{D}$. At
this point, the vectors $\overrightarrow{a},\overrightarrow{b},\ldots,\overrightarrow{C},\overrightarrow{D}$
and the permutations $\eta_{1},\eta_{2},\eta_{3}$ are known, and
the number of options for $\eta_{4}$ is determined by the shapes
of $\overrightarrow{a},\overrightarrow{b}$. This argument will be
generalized in $\mathsection$\ref{sec:Estimates-on} for arbitrary
words, where, instead of a $4\times4$ matrix, we will have a $2r\times2r$
matrix and, each time we choose $\eta_{1},\ldots,\eta_{k}$, the $k+1$-st
row is revealed, allowing us to proceed by induction. 
\end{rem}

\section{\label{sec:Reduction-of-Theorem}Rewriting Theorem \ref{thm:abs_value_of_trace}
using Weingarten calculus}

In this section, we rewrite the expression $\mathbb{E}\left(\left|\tr\bigwedge\nolimits ^{m}w(X_{1},\ldots,X_{r})\right|^{2}\right)$
of Theorem \ref{thm:abs_value_of_trace} as a finite sum of Weingarten
functions. 

Let $\ell,m,d,w$ be as in Theorem \ref{thm:abs_value_of_trace}.
We may assume that $w$ is \emph{cyclically reduced}, i.e., it does
not contain a subword of the form $x_{j}x_{j}^{-1}$ and the first
and last letters of $w$ are not inverse of each other. For $u\in[\ell]$,
let 
\[
w(u)=\begin{cases}
a & \text{if the \ensuremath{u^{\mathrm{th}}} letter of \ensuremath{w} is \ensuremath{x_{a}}}\\
-a & \text{if the \ensuremath{u^{\mathrm{th}}} letter of \ensuremath{w} is \ensuremath{x_{a}^{-1}}}
\end{cases}.
\]
If we denote $x_{-a}=x_{a}^{-1}$, then $w=\prod_{u}x_{w(u)}$. We
write $w^{-1}$ for the inverse word,
\begin{equation}
w^{-1}:=x_{-w(\ell)}x_{-w(\ell-1)}\cdots x_{-1}.\label{eq:inverse word}
\end{equation}
We start by noting that
\begin{align}
\mathbb{E}\left(\left|\tr\left(\bigwedge\nolimits ^{m}w(X_{1},\ldots,X_{r})\right)\right|^{2}\right) & =\mathbb{E}\left(\tr\left(\bigwedge\nolimits ^{m}w(X_{1},\ldots,X_{r})\right)\cdot\overline{\tr\left(\bigwedge\nolimits ^{m}w(X_{1},\ldots,X_{r})\right)}\right)\label{eq:expected value of square of trace}\\
 & =\mathbb{E}\left(\tr\left(\bigwedge\nolimits ^{m}w(X_{1},\ldots,X_{r})\right)\cdot\tr\left(\bigwedge\nolimits ^{m}w^{-1}(X_{1},\ldots,X_{r})\right)\right).\nonumber 
\end{align}
Define $\widetilde{T}\in\Sym([\ell]\times[m])$ by 
\begin{equation}
\widetilde{T}(u,k)=\begin{cases}
(u+1,k) & u\neq\ell\\
(1,k) & u=\ell
\end{cases}.\label{eq:T_tilde}
\end{equation}
Recall that $\mathcal{I}_{m,d}=\{a_{1}<\ldots<a_{m}:a_{i}\in[d]\}$.
We have 
\begin{align}
\tr\left(\bigwedge\nolimits ^{m}w(X_{1},\ldots,X_{r})\right) & =\sum_{\overrightarrow{a}\in\mathcal{I}_{m,d}}\sum_{\pi\in S_{m}}(-1)^{\pi}\prod_{k=1}^{m}w(X_{1},\ldots,X_{r})_{a_{k},a_{\pi(k)}}\nonumber \\
 & =\sum_{\overrightarrow{a}\in\mathcal{I}_{m,d}}\sum_{\pi\in S_{m}}(-1)^{\pi}\prod_{k=1}^{m}\sum_{\substack{f_{k}:[\ell+1]\rightarrow[d]\\
f_{k}(1)=a_{k},f_{k}(\ell+1)=a_{\pi(k)}
}
}\prod_{u=1}^{\ell}\left(X_{w(u)}\right)_{f_{k}(u),f_{k}(u+1)}\nonumber \\
 & =\sum_{\overrightarrow{a}\in\mathcal{I}_{m,d}}\sum_{\pi\in S_{m}}(-1)^{\pi}\sum_{\substack{f:[\ell+1]\times[m]\rightarrow[d]\\
f(1,k')=a_{k'},f(\ell+1,k')=a_{\pi(k')},\forall k'
}
}\prod_{(u,k)\in[\ell]\times[m]}\left(X_{w(u)}\right)_{f(u,k),f(u+1,k)}\nonumber \\
 & =\sum_{\pi\in S_{m}}(-1)^{\pi}\sum_{\substack{f:[\ell+1]\times[m]\rightarrow[d]\\
f(\ell+1,k')=f(1,\pi(k')),\forall k'\\
f(1,-)\text{ increasing}
}
}\prod_{(u,k)\in[\ell]\times[m]}\left(X_{w(u)}\right)_{f(u,k),f(u+1,k)}\nonumber \\
 & =\sum_{\pi\in\Sym(\left\{ \ell\right\} \times[m])}(-1)^{\pi}\sum_{\substack{F:[\ell]\times[m]\rightarrow[d]\\
F(1,-)\text{ increasing}
}
}\prod_{(u,k)\in[\ell]\times[m]}\left(X_{w(u)}\right)_{F(u,k),F(\widetilde{T}\pi(u,k))},\label{eq:reformulation of trace}
\end{align}
where in the last equality we use the natural embedding $\Sym(\left\{ \ell\right\} \times[m])\hookrightarrow\Sym([\ell]\times[m])$
obtained by acting trivially on $[\ell-1]\times[m]$. Applying this
to $w^{-1}$, we get 
\begin{equation}
\overline{\tr\bigwedge\nolimits ^{m}w(X_{1},\ldots,X_{r})}=\sum_{\pi'\in\Sym(\left\{ \ell\right\} \times[m])}(-1)^{\pi'}\sum_{\substack{F':[\ell]\times[m]\rightarrow[d]\\
F'(1,-)\text{ increasing}
}
}\prod_{(u,k)\in[\ell]\times[m]}\left(X_{w^{-1}(u)}\right)_{F'(u,k),F'(\widetilde{T}\pi'(u,k))}.\label{eq:simplified trace for w^-1}
\end{equation}
Set $\Omega=[2]\times[\ell]\times[m]$, $\Omega_{s,u}=\left\{ s\right\} \times\left\{ u\right\} \times[m]$,
and for $\gamma\in\Omega$, define
\[
\widetilde{w}(\gamma)=\begin{cases}
w(u) & \gamma=(1,u,k)\\
w^{-1}(u) & \gamma=(2,u,k)
\end{cases}.
\]
Define $T\in\Sym(\Omega)$ by 
\begin{equation}
T(s,u,k):=(s,\widetilde{T}(u,k)).\label{eq:def of T}
\end{equation}
By combining (\ref{eq:reformulation of trace}) and (\ref{eq:simplified trace for w^-1}),
we get
\begin{equation}
\left|\tr\bigwedge\nolimits ^{m}w(X_{1},\ldots,X_{r})\right|^{2}=\sum_{(\pi,\pi')\in\prod_{s=1}^{2}\Sym(\Omega_{s,\ell})}(-1)^{\pi\pi'}\sum_{\substack{F:\Omega\rightarrow[d]\\
F(1,1,-)\,\text{increasing}\\
F(2,1,-)\,\text{increasing}
}
}\prod_{\gamma\in\Omega}\left(X_{\widetilde{w}(\gamma)}\right)_{F(\gamma),F(T\pi\pi'(\gamma))}.\label{eq:square of trace simplified}
\end{equation}
The map $\pi\mapsto T\pi T^{-1}$ is an isomorphism $\Sym(\Omega_{s,\ell})\overset{\simeq}{\rightarrow}\Sym(\Omega_{s,1})$,
for $s\in[2]$. Hence, 
\begin{align}
\left|\tr\bigwedge\nolimits ^{m}w(X_{1},\ldots,X_{r})\right|^{2} & =\sum_{(\pi,\pi')\in\Sym(\Omega_{1,1})\times\Sym(\Omega_{2,1})}(-1)^{\pi\pi'}\sum_{\substack{F:\Omega\rightarrow[d]\\
F(1,1,-)\,\text{increasing}\\
F(2,1,-)\,\text{increasing}
}
}\prod_{\gamma\in\Omega}\left(X_{\widetilde{w}(\gamma)}\right)_{F(\gamma),F(\pi\pi'T(\gamma))}\nonumber \\
 & =\sum_{(\pi,\pi')\in\Sym(\Omega_{1,1})\times\Sym(\Omega_{2,1})}(-1)^{\pi\pi'}\sum_{\substack{F:\Omega\rightarrow[d]\\
F(1,1,-)\,\text{increasing}\\
F(2,1,-)\,\text{increasing}
}
}\prod_{\gamma\in\Omega}\left(X_{\widetilde{w}(\gamma)}\right)_{F(\gamma),F((\pi\pi')^{-1}T(\gamma))}\nonumber \\
 & =\sum_{(\pi,\pi')\in\Sym(\Omega_{1,1})\times\Sym(\Omega_{2,1})}(-1)^{\pi\pi'}\sum_{\substack{F:\Omega\rightarrow[d]\\
F\circ\pi(1,1,-)\,\text{increasing}\\
F\circ\pi'(2,1,-)\,\text{increasing}
}
}\prod_{\gamma\in\Omega}\left(X_{\widetilde{w}(\gamma)}\right)_{F(\pi\pi'\gamma),F(T(\gamma))},\label{eq:very simplified formula for trace}
\end{align}
where, in the last equality, we replaced $F$ by $F\circ\left(\pi'\pi\right)^{-1}$. 

Let $\Phi=(A_{1},\ldots,A_{r},B_{1},\ldots,B_{r})$ be the partition
given by 
\begin{equation}
A_{i}=\left\{ (s,u,k)\mid\widetilde{w}(s,u,k)=i\right\} \quad B_{i}=\left\{ (s,u,k)\mid\widetilde{w}(s,u,k)=-i\right\} .\label{eq:A_i and B_i}
\end{equation}
For each $(\pi,\pi')\in\Sym(\Omega_{1,1})\times\Sym(\Omega_{2,1})$,
set
\begin{equation}
Z_{\pi,\pi'}:=\left\{ (F,\varSigma):\substack{F:\Omega\rightarrow[d],\varSigma\in S_{\Phi}\\
F\circ\pi(1,1,-),F\circ\pi'(2,1,-)\text{ increasing}\\
F\circ T=F\circ\pi\pi'\circ\Sigma
}
\right\} .\label{eq:the set Z_pi,pi'}
\end{equation}
The sets $Z_{\pi,\pi'}$ are disjoint. We denote
\begin{equation}
Z:=\bigcup_{\pi,\pi'}Z_{\pi,\pi'}.\label{eq:the set Z}
\end{equation}

\begin{rem}
\label{rem:pi_F}Note that we have a map $Z\rightarrow\Sym(\Omega_{1,1})\times\Sym(\Omega_{2,1})$
sending $(F,\varSigma)$ to the unique pair $(\pi_{F},\pi'_{F})$
such that $(F,\varSigma)\in Z_{\pi_{F},\pi'_{F}}$.
\end{rem}

Rewriting (\ref{eq:very simplified formula for trace}) using Weingarten
calculus (Corollary \ref{cor:Weingarten coordinate free- multidimensional}),
we have:
\begin{prop}
\label{prop:Main result-reduction to Weingarten}Let $w\in F_{r}$
be a cyclically reduced word. Then:
\begin{equation}
\mathbb{E}\left(\left|\tr\left(\bigwedge\nolimits ^{m}w(X_{1},\ldots,X_{r})\right)\right|^{2}\right)=\sum_{(F,\varSigma)\in Z}(-1)^{\pi_{F}\pi'_{F}}\widetilde{\Wg}(\varSigma^{2}).\label{eq:reduction to Weingarten}
\end{equation}
\end{prop}

\section{\label{sec:Estimating-the-contribution}Estimating the contribution
of a single orbit in $Z$}

In this section we introduce an action of $H:=\prod_{(s,u)\in[2]\times[\ell]}\Sym(\Omega_{s,u})$
on $Z$, and estimate (\ref{eq:reduction to Weingarten}) restricted
to each $H$-orbit. The action can be described as follows. 

For every $(s,u)\in[2]\times[\ell]$, the group $\Sym(\Omega_{s,u})$
acts on $Z$ in the following way: if $u\neq1$, the action of $\pi_{s,u}\in\Sym(\Omega_{s,u})$
is 
\begin{equation}
\pi_{s,u}.(F,\varSigma)=(F\circ\pi_{s,u}^{-1},\pi_{s,u}\circ\Sigma\circ T^{-1}\pi_{s,u}^{-1}T).\label{eq:most of the action}
\end{equation}
If $s\in[2]$ and $\pi_{s,1}\in\Sym(\Omega_{s,1})$, then 
\begin{equation}
\pi_{s,1}.(F,\varSigma)=(F\circ\pi_{s,1}^{-1},\varSigma\circ T^{-1}\pi_{s,1}^{-1}T).\label{eq:action of u=00003D1}
\end{equation}
The above group actions commute, which gives rise to an action of
$H$. Note that $(\pi_{1,1},\pi_{2,1}).Z_{\pi,\pi'}=Z_{\pi_{1,1}\pi,\pi_{2,1}\pi'}$.
If $u\neq1$, then $\pi_{s,u}.(Z_{\pi,\pi'})=Z_{\pi,\pi'}$. 
\begin{defn}
\label{def:new multiplication}For each $u,v\in[\ell]$, we define
$*:$$\Sym(\Omega_{s,u})\times\Sym(\Omega_{s,v})\rightarrow\Sym(\Omega_{s,v})$
by 
\begin{equation}
\pi_{s,u}*\pi_{s,v}:=T^{v-u}\pi_{s,u}T^{u-v}\pi_{s,v}.\label{eq:new action on premutations}
\end{equation}
Note that $*$ is associative. 
\end{defn}

Let $h:=\prod_{(s,u)}\pi_{s,u}\in H$ and denote $\overline{h}:=\prod_{(s,u)\neq(1,1),(2,1)}\pi_{s,u}$.
Then $h.\varSigma=\overline{h}\circ\varSigma\circ T^{-1}h^{-1}T$.
Since $\widetilde{\Wg}$ is invariant under conjugation in $H$,
\[
\widetilde{\Wg}\left(\left(h.(\varSigma)\right)^{2}\right)=\widetilde{\Wg}\left(\Psi_{h}\circ\varSigma\circ\Psi_{h}\circ\varSigma\right),
\]
where $\Psi_{h}=T^{-1}h^{-1}T\overline{h}\in H$. On each $\Omega_{s,u}$,
$\Psi_{h}$ has the following form:
\begin{lem}
\label{lem:Structure of =00005Cpsi_h}We have
\[
\Psi_{h}|_{\Omega_{s,u}}=\begin{cases}
T^{-1}\pi_{s,2}^{-1}T & \text{if }u=1\\
\pi_{s,u+1}^{-1}*\pi_{s,u} & \text{if }u\neq1,\ell\\
\pi_{s,1}^{-1}*\pi_{s,\ell} & \text{if }u=\ell.
\end{cases}
\]
\end{lem}

\begin{cor}
\label{cor:analyzing H-orbits}Let $(\widehat{F},\widehat{\varSigma})$
be a representative of an $H$-orbit $\mathcal{O}_{(\widehat{F},\widehat{\varSigma})}$,
with $(\pi_{\widehat{F}},\pi'_{\widehat{F}})=(\mathrm{Id},\mathrm{Id})$.
Then:
\begin{equation}
\frac{1}{\left|\mathcal{O}_{(\widehat{F},\widehat{\varSigma})}\right|}\sum_{(F,\varSigma)\in\mathcal{O}_{(\widehat{F},\widehat{\varSigma})}}(-1)^{\pi_{F}\pi'_{F}}\widetilde{\Wg}(\varSigma^{2})=\frac{1}{m!^{2\ell}}\prod_{i=1}^{r}\sum_{\substack{h_{i}\in\prod_{\widetilde{w}=i}\Sym(\Omega_{s,u})\\
h'_{i}\in\prod_{\widetilde{w}=-i}\Sym(\Omega_{s,u})
}
}(-1)^{h_{i}h'_{i}}\Wg\left(h_{i}\widehat{\varSigma}|_{B_{i}}h'_{i}\widehat{\varSigma}|_{A_{i}}\right).\label{eq:cor 5.3}
\end{equation}
\end{cor}

\begin{proof}
For each $h=\prod_{(s,u)}\pi_{s,u}\in H$, write $\nu(h):=(-1)^{\pi_{1,1}\pi_{2,1}}$.
Consider the bijection $\psi:H\rightarrow H$, $\psi\left(\prod_{(s,u)}\pi_{s,u}\right)=\prod_{(s,u)}\theta_{s,u}$,
where, for $s=1,2$,
\[
(\theta_{s,2},\ldots,\theta_{s,\ell},\theta_{s,1})=(\pi_{s,2},\pi_{s,2}*\pi_{s,3},\ldots,\pi_{s,2}*\cdots*\pi_{s,\ell},\pi_{s,2}*\cdots*\pi_{s,\ell}*\pi_{s,1}),
\]
and observe that $\nu(\psi(h))=(-1)^{h}=(-1)^{T^{-1}h^{-1}T}$. Further
note that
\[
\left(\pi_{s,2}*\cdots*\pi_{s,u+1}\right)^{-1}*\pi_{s,2}*\cdots*\pi_{s,u}=T^{-1}\pi_{s,u+1}^{-1}T,
\]
and hence $\Psi_{\psi(h)}=\prod_{(s,u)}T^{-1}\pi_{s,u}^{-1}T$. Changing
variables using $\psi$, the left hand side of (\ref{eq:cor 5.3})
is:
\begin{align*}
\frac{1}{\left|H\right|}\sum_{h\in H}\nu(h)\widetilde{\Wg}\left(\Psi_{h}\circ\widehat{\varSigma}\circ\Psi_{h}\circ\widehat{\varSigma}\right) & =\frac{1}{m!^{2\ell}}\sum_{h\in H}\nu(\psi(h))\widetilde{\Wg}\left(\prod_{(s,u)}T^{-1}\pi_{s,u}^{-1}T\circ\widehat{\varSigma}\circ\prod_{(s,u)}T^{-1}\pi_{s,u}^{-1}T\circ\widehat{\varSigma}\right)\\
 & =\frac{1}{m!^{2\ell}}\sum_{h\in H}(-1)^{h}\widetilde{\Wg}\left(\prod_{(s,u)}\pi_{s,u}\circ\widehat{\varSigma}\circ\prod_{(s,u)}\pi_{s,u}\circ\widehat{\varSigma}\right)\\
 & =\frac{1}{m!^{2\ell}}\sum_{h\in H}(-1)^{h}\prod_{i=1}^{r}\Wg\left(\prod_{(s,u):\widetilde{w}=i}\pi_{s,u}\widehat{\varSigma}|_{B_{i}}\prod_{(s,u):\widetilde{w}=-i}\pi_{s,u}\widehat{\varSigma}|_{A_{i}}\right),
\end{align*}
where in each line above, $h=\prod_{(s,u)}\pi_{s,u}$.
\end{proof}
\begin{cor}
\label{cor:reduction to L-orbits on Z}Set $\ell_{i}:=\frac{\left|A_{i}\right|}{m}$
for each $i\in[r]$. Then the following holds: 
\begin{equation}
\mathbb{E}\left(\left|\tr\left(\bigwedge\nolimits ^{m}w(X_{1},\ldots,X_{r})\right)\right|^{2}\right)\leq\frac{\left|Z\right|}{m!^{\ell}}\prod_{i=1}^{r}\frac{1}{d\cdots(d-m\ell_{i}+1)}.\label{eq:reduction to the size of Z}
\end{equation}
\end{cor}

\begin{proof}
By Proposition \ref{prop:Main result-reduction to Weingarten}, Corollary
\ref{cor:analyzing H-orbits}, Equation (\ref{eq:Weingarten function}),
Lemma \ref{lem:primitive character bounds}, and by Equation (\ref{eq:formula for dimension}),
\begin{align}
\mathbb{E}\left(\left|\tr\left(\bigwedge\nolimits ^{m}w(X_{1},\ldots,X_{r})\right)\right|^{2}\right)= & \sum_{(\widehat{F},\widehat{\varSigma})\in Z}(-1)^{\pi_{\widehat{F}}\pi'_{\widehat{F}}}\widetilde{\Wg}(\widehat{\varSigma}^{2})\nonumber \\
= & \sum_{(\widehat{F},\widehat{\varSigma})\in Z}\frac{1}{\left|\mathcal{O}_{(\widehat{F},\widehat{\varSigma})}\right|}\sum_{(F,\varSigma)\in\mathcal{O}_{(\widehat{F},\widehat{\varSigma})}}(-1)^{\pi_{F}\pi'_{F}}\widetilde{\Wg}(\varSigma^{2})\nonumber \\
\leq & \sum_{(\widehat{F},\widehat{\varSigma})\in Z}\frac{1}{m!^{2\ell}}\prod_{i=1}^{r}\left|\sum_{\substack{h_{i}\in\prod_{\widetilde{w}=i}\Sym(\Omega_{s,u})\\
h'_{i}\in\prod_{\widetilde{w}=-i}\Sym(\Omega_{s,u})
}
}(-1)^{h_{i}h'_{i}}\Wg\left(h_{i}\widehat{\varSigma}|_{B_{i}}h'_{i}\widehat{\varSigma}|_{A_{i}}\right)\right|\nonumber \\
\leq & \sum_{(\widehat{F},\widehat{\varSigma})\in Z}\frac{1}{m!^{2\ell}}\prod_{i=1}^{r}\frac{m!^{2\ell_{i}}}{(m\ell_{i})!^{2}}\sum_{\lambda\vdash m\ell_{i}:\chi_{\lambda}\subseteq\mathrm{Ind}_{\mathrm{S}_{m}^{\ell_{i}}}^{S_{m\ell_{i}}}(\sgn)}\langle\chi_{\lambda},\sgn\rangle_{\mathrm{S}_{m}^{\ell_{i}}}\frac{\chi_{\lambda}(1)^{2}}{\rho_{\lambda}(1)}\nonumber \\
= & \frac{\left|Z\right|}{m!^{\ell}}\prod_{i=1}^{r}\frac{m!^{\ell_{i}}}{(m\ell_{i})!}\sum_{\lambda\vdash m\ell_{i}:\chi_{\lambda}\subseteq\mathrm{Ind}_{\mathrm{S}_{m}^{\ell_{i}}}^{S_{m\ell_{i}}}(\sgn)}\frac{\langle\chi_{\lambda},\sgn\rangle_{\mathrm{S}_{m}^{\ell_{i}}}\chi_{\lambda}(1)}{\prod_{(a,b)\in\lambda}(d+b-a)}.\label{eq:Z.to.orbit}
\end{align}
Note that the irreducible characters $\chi_{\lambda}$ in $\mathrm{Ind}_{\mathrm{S}_{m}^{\ell_{i}}}^{S_{m\ell_{i}}}(\sgn)$
correspond to Young diagrams $\lambda\vdash m\ell_{i}$ with at most
$\ell_{i}$ columns. If the columns of $\lambda$ are of lengths $j_{1}\geq\ldots\geq j_{\ell_{i}}$
then
\begin{equation}
\prod_{(a,b)\in\lambda}(d+b-a)\geq d\cdots(d-j_{1}+1)\cdot d\cdots(d-j_{2}+1)\cdot\ldots\cdot d\cdots(d-j_{\ell_{i}}+1)\geq d\cdots(d-m\ell_{i}+1).\label{eq:(5.6)}
\end{equation}
Combining (\ref{eq:Z.to.orbit}) with (\ref{eq:(5.6)}) implies the
corollary. 
\end{proof}

\section{\label{sec:Estimates-on}Estimates on $\left|Z\right|$}

In this section we give upper bounds on $\left|Z\right|$, defined
in (\ref{eq:the set Z}). We first set some notation. For each $0\neq i,j\in[-r,r]$,
set 
\[
R_{i}:=\left\{ \gamma\in\Omega:\widetilde{w}\circ T^{-1}(\gamma)=i\right\} =\begin{cases}
T(A_{i}) & i>0\\
T(B_{-i}) & i<0
\end{cases},
\]
\[
C_{j}:=\left\{ \gamma\in\Omega:\widetilde{w}(\gamma)=-j\right\} =\begin{cases}
B_{j} & j>0\\
A_{-j} & j<0
\end{cases},
\]
\[
V_{ij}:=\left\{ \gamma\in\Omega:\widetilde{w}\circ T^{-1}(\gamma)=i,\widetilde{w}(\gamma)=-j\right\} =R_{i}\cap C_{j}.
\]
Following Remark \ref{rem:generalization to arbitrary words}, it
is helpful to picture a $2r\times2r$ matrix, whose $(i,j)$-th entry
is the set $V_{ij}$, with $R_{-r},\ldots,R_{r}$ correspond to rows,
and $C_{-r},\ldots,C_{r}$ correspond to columns. Denote 
\begin{equation}
\ell_{i,j}:=\frac{\left|V_{ij}\right|}{m}\text{ and }\ell_{i}:=\frac{\left|R_{i}\right|}{m}=\frac{\left|C_{i}\right|}{m}.\label{eq:l_ij}
\end{equation}
Observe that $\ell_{i}=\ell_{-i}$, $\ell_{i,j}=\ell_{j,i}$ and note
that $\ell_{i}=\frac{\left|A_{i}\right|}{m}$ if $i>0$, so that (\ref{eq:l_ij})
extends the definition of $\ell_{i}$ in Corollary \ref{cor:reduction to L-orbits on Z}.
For each $0\neq j\in[-r,r]$ set
\[
C_{j}^{+}:=\bigcup_{i<j}V_{ij},\,\,C_{+}:=\bigcup_{j}C_{j}^{+}.
\]
For each $i\in[r]$ and each $\varSigma\in S_{\Phi}$, denote $\eta_{i}:=T\circ\left(\varSigma^{-1}\right)|_{B_{i}}$
and $\eta_{-i}:=T\circ\left(\varSigma^{-1}\right)|_{A_{i}}$. Notice
that $\eta_{i}(C_{i})=R_{i}$ for all $i$. Define the following sets:
\begin{equation}
W':=\left\{ \left(F:\Omega\rightarrow[d],\varSigma\in S_{\Phi}\right):F\circ T=F\circ\Sigma\right\} ,\label{eq:W'}
\end{equation}
and
\[
W:=\left\{ \left(F,\varSigma\right)\in W':F(s,1,-)\text{ is one-to-one\,\,\,}\forall s\in[2]\right\} .
\]

\begin{prop}
\label{prop:estimates on Z}We have
\[
\left|Z\right|=\left|W\right|\leq\left|W'\right|\leq\binom{d+m\ell}{m\ell}(m\ell)!\prod_{0\neq k=-r}^{r}\frac{\left(m\ell_{k}\right)!}{\left(\sum_{i<k}m\ell_{i,k}\right)!}.
\]
\end{prop}

\begin{proof}
The map $\left(F,\varSigma\right)\mapsto\left(F\circ\pi_{F}\pi'_{F},\varSigma\circ T{}^{-1}\pi_{F}\pi'_{F}T\right)$
is a bijection between $Z$ and $W$, giving the first equality. Clearly,
$\left|W\right|\leq\left|W'\right|$. 

In order to prove the last inequality, we use the map $\Phi_{+}:W'\rightarrow\left\{ f:C_{+}\rightarrow[d]\right\} $,
sending $(F,\varSigma)\in W'$ to $F|_{C_{+}}$. We estimate $\left|W'\right|$
by analyzing the fibers of $\Phi_{+}$. Let $f\in\Phi_{+}(W')$ and
suppose it has a shape $\nu_{+}$ (see Definition \ref{def:shape of a function}).
We write $\nu_{j,+}$ for the shapes of $f|_{C_{j}^{+}}$. We reveal
$(F,\varSigma)\in\Phi_{+}^{-1}(f)$ row by row, starting with the
$-r$-th row $R_{-r}$ and making sure that, in each step, $F\circ T|_{T^{-1}(R_{k})}=F\circ\varSigma|_{T^{-1}(R_{k})}$,
or, equivalently, $F|_{R_{k}}=F\circ\eta_{k}^{-1}|_{R_{k}}$.
\begin{enumerate}
\item There are at most $(m\ell_{-r})!$ options for $\eta_{-r}$. Note
that $R_{-r}\subseteq C_{+}$ and hence by (\ref{eq:W'}), the choice
of $\eta_{-r}$ determines $F|_{C_{-r}}$. At this point, $F|_{R_{-r+1}}$
is determined as well. 
\item Note that $C_{-r+1}^{+}=V_{-r,-r+1}$. There are at most:
\begin{enumerate}
\item $\binom{m\ell_{-r+1}}{m\ell_{-r,-r+1}}$ options for the sets $\eta_{-r+1}(C_{-r+1}^{+})$
and $\eta_{-r+1}(C_{-r+1}\backslash C_{-r+1}^{+})$. 
\item $\left(m(\ell_{-r+1}-\ell_{-r,-r+1})\right)!$ options for $\eta_{-r+1}|_{C_{-r+1}\backslash C_{-r+1}^{+}}:C_{-r+1}\backslash C_{-r+1}^{+}\rightarrow\eta_{-r+1}(C_{-r+1}\backslash C_{-r+1}^{+})$. 
\item $\left(\nu_{-r+1,+}\right)!$ options for $\eta_{-r+1}:C_{-r+1}^{+}\rightarrow\eta_{-r+1}(C_{-r+1}^{+})$.
\end{enumerate}
\item More generally, assume, by induction, that we have fixed $\left(\eta_{i}\right)_{i<k}$,
and, thus, we have already determined $F|_{R_{i}}$ for $i\leq k$,
$F|_{C_{i}}$ for $i<k$, and $F|_{C_{+}}$. Then there are at most:
\begin{enumerate}
\item $\binom{m\ell_{k}}{\sum_{i<k}m\ell_{i,k}}$ options for the sets $\eta_{k}(C_{k}^{+})$
and $\eta_{k}(C_{k}\backslash C_{k}^{+})$. 
\item $\left(m(\ell_{k}-\sum_{i<k}\ell_{i,k})\right)!$ options for $\eta_{k}|_{C_{k}\backslash C_{k}^{+}}:C_{k}\backslash C_{k}^{+}\rightarrow\eta_{k}(C_{k}\backslash C_{k}^{+})$. 
\item $\left(\nu_{k,+}\right)!$ options for $\eta_{k}|_{C_{k}^{+}}:C_{k}^{+}\rightarrow\eta_{k}(C_{k}^{+})$. 
\end{enumerate}
\end{enumerate}
After choosing $\eta_{-r},\ldots,\eta_{r}$, we have determined $F$.
Furthermore, since $\sum_{0\neq k=-r}^{r}\nu_{k,+}=\nu_{+}$, we have
$\prod_{0\neq k=-r}^{r}\left(\nu_{k,+}\right)!\leq\nu_{+}!$. Hence,
\begin{align}
\left|\Phi_{+}^{-1}(f)\right| & \leq\prod_{0\neq k=-r}^{r}\left(\binom{m\ell_{k}}{\sum_{i<k}m\ell_{i,k}}\left(m(\ell_{k}-\sum_{i<k}\ell_{i,k})\right)!\left(\nu_{k,+}\right)!\right)\nonumber \\
 & =\prod_{0\neq k=-r}^{r}\left(\nu_{k,+}\right)!\frac{\left(m\ell_{k}\right)!}{\left(\sum_{i<k}m\ell_{i,k}\right)!}\leq\nu_{+}!\prod_{0\neq k=-r}^{r}\frac{\left(m\ell_{k}\right)!}{\left(\sum_{i<k}m\ell_{i,k}\right)!}.\label{eq:bound on the fibers}
\end{align}
Since $\left|C_{+}\right|=m\ell$, we have 
\begin{equation}
\left|\left\{ f\in\Phi_{+}(W'):f\text{ is of shape }\nu_{+}\right\} \right|\leq\frac{(m\ell)!}{\nu_{+}!},\label{eq:bound on the size of the Borel}
\end{equation}
and there are at most $\binom{d+m\ell}{m\ell}$ possible shapes $\nu_{+}$.
Combining (\ref{eq:bound on the fibers}) and (\ref{eq:bound on the size of the Borel})
we conclude:
\begin{align*}
\left|W'\right| & \leq\sum_{\nu_{+}}\left|\left\{ f\in\Phi_{+}(W'):f\text{ is of shape }\nu_{+}\right\} \right|\cdot\left|\Phi_{+}^{-1}(f)\right|\\
 & \leq\sum_{\nu_{+}}\frac{(m\ell)!}{\nu_{+}!}\cdot\nu_{+}!\prod_{0\neq k=-r}^{r}\frac{\left(m\ell_{k}\right)!}{\left(\sum_{i<k}m\ell_{i,k}\right)!}\leq\binom{d+m\ell}{m\ell}(m\ell)!\prod_{0\neq k=-r}^{r}\frac{\left(m\ell_{k}\right)!}{\left(\sum_{i<k}m\ell_{i,k}\right)!}.\qedhere
\end{align*}
\end{proof}

\section{\label{sec:Proof-of-main theorems}Proof of Theorems \ref{thm:main_thm_fundemental_rep}
and \ref{thm:abs_value_of_trace}}

In this section we use the results of Sections \ref{sec:Reduction-of-Theorem},
\ref{sec:Estimating-the-contribution} and \ref{sec:Estimates-on}
to prove Theorems \ref{thm:main_thm_fundemental_rep} and \ref{thm:abs_value_of_trace}.
We end the section with the proof of Theorem \ref{thm:LST theorem}.
\begin{proof}[Proof of Theorem \ref{thm:abs_value_of_trace}]
Assume that $d=am$ for $a\geq\ell\geq2$. By (\ref{eq:Striling on binomial coefficients}),
we have:
\begin{equation}
\binom{d}{m\ell}=\binom{am}{m\ell}\geq\frac{a^{m\ell}}{\ell^{m\ell}},\label{eq:doverm}
\end{equation}
\begin{equation}
\binom{d+m\ell}{m\ell}\leq\left(\frac{a+\ell}{\ell}\right)^{m\ell}e^{m\ell}\leq\frac{a^{m\ell}(2e)^{m\ell}}{\ell^{m\ell}}.\label{eq:dover_ml}
\end{equation}
We remind the reader the definition of $\ell_{i}$ and $\ell_{i,j}$
in (\ref{eq:l_ij}). Concretely, for each $0\neq i,j\in[-r,r]$, $\ell_{i}$
is the combined number of appearances of the letter $x_{i}$ (with
the convention that $x_{-i}=x_{i}^{-1}$) in $w$ and $w^{-1}$, and
$\ell_{i,j}$ is the combined number of appearances of the string
``$x_{i}x_{j}^{-1}$'' in $w$ and in $w^{-1}$. In particular,
we have $\sum_{i=1}^{r}\ell_{i}=\ell$, $\ell_{i,i}=0$ and $\sum_{0\neq i\in[-r,r]}\ell_{i,k}=\ell_{k}$
and therefore:
\begin{equation}
\prod_{i=1}^{r}d\cdots(d-m\ell_{i}+1)\geq d\cdots(d-m\ell+1)\text{\,\,\, and \,\,\,}\frac{\left(m\ell_{k}\right)!}{\left(\sum_{i<k}m\ell_{i,k}\right)!\left(\sum_{i>k}m\ell_{i,k}\right)!}=\binom{m\ell_{k}}{\sum_{i>k}m\ell_{i,k}}\leq2^{m\ell_{k}}.\label{eq:auxilary computations}
\end{equation}
By Corollary \ref{cor:reduction to L-orbits on Z}, Proposition \ref{prop:estimates on Z}
and by (\ref{eq:auxilary computations}), (\ref{eq:doverm}) and (\ref{eq:dover_ml}),
we obtain:
\begin{align*}
 & \mathbb{E}\left(\left|\tr\left(\bigwedge\nolimits ^{m}w(X_{1},\ldots,X_{r})\right)\right|^{2}\right)\leq\frac{\left|Z\right|}{m!^{\ell}}\prod_{i=1}^{r}\frac{1}{d\cdots(d-m\ell_{i}+1)}\\
\leq & \binom{d+m\ell}{m\ell}\cdot\frac{(m\ell)!}{\prod_{i=1}^{r}d\cdots(d-m\ell_{i}+1)}\cdot\frac{1}{m!^{\ell}}\cdot\prod_{0\neq k=-r}^{r}\frac{\left(m\ell_{k}\right)!}{\left(\sum_{i<k}m\ell_{i,k}\right)!}\\
\leq & \binom{d+m\ell}{m\ell}\cdot\frac{(m\ell)!}{d\cdots(d-m\ell+1)}\cdot\frac{\prod_{0\neq k=-r}^{r}\left(\sum_{i>k}m\ell_{i,k}\right)!}{m!^{\ell}}\cdot\prod_{0\neq k=-r}^{r}\frac{\left(m\ell_{k}\right)!}{\left(\sum_{i<k}m\ell_{i,k}\right)!\left(\sum_{i>k}m\ell_{i,k}\right)!}\\
\leq & \binom{d+m\ell}{m\ell}\cdot\binom{d}{m\ell}^{-1}\cdot\frac{(m\ell)!}{m!^{\ell}}\cdot\prod_{0\neq k=-r}^{r}2^{m\ell_{k}}\leq(2e)^{m\ell}\ell^{m\ell}\cdot2^{2m\ell}\leq(8e\ell)^{m\ell}\leq(22\ell)^{m\ell}.
\end{align*}
Finally, note that if $d\geq(22\ell)^{\ell}m$ then $(22\ell)^{m\ell}\leq(\frac{d}{m})^{m}\leq\binom{d}{m}$.
\end{proof}
We now turn to the proof of Theorem \ref{thm:main_thm_fundemental_rep}.
We first deal with the case when the rank is bounded (and prove Conjecture
\ref{conj:Fourier coefficients of compact Lie groups} in this case)
and then prove Theorem \ref{thm:main_thm_fundemental_rep} in the
unbounded case. 
\begin{defn}
\label{def:Concatination of words}Given $w_{1}\in F_{r_{1}}$ and
$w_{2}\in F_{r_{2}}$, we denote by $w_{1}*w_{2}\in F_{r_{1}+r_{2}}$
their \emph{concatenation. For example, if $w=[x,y]$, then $w*w=[x,y]\cdot[z,w]$.}
\end{defn}

We remind the reader the for a compact group $G$, and a word $w\in F_{r}$,
we denote by $\tau_{w,G}:=(w_{G})_{*}(\mu_{G}^{r})$ the word measure
associated to $w$ and $G$, and the Fourier coefficient of $\tau_{w,G}$
at $\rho\in\mathrm{Irr}(G)$ is $a_{w,G,\rho}:=\int_{G^{r}}\rho(w(x_{1},\ldots,x_{r}))\mu_{G}^{r}=\int_{G}\rho(y)\tau_{w,G}$.
If $G$ is a compact connected semisimple Lie group, by \cite{Bor83},
the map $w_{G}:G^{r}\rightarrow G$ is a submersion outside a proper
subvariety in $G^{r}$. It follows that in this case, or e.g.~when
$G$ is a finite group, $\tau_{w,G}$ is absolutely continuous with
respect to $\mu_{G}$, and we can write $\tau_{w,G}=f_{w,G}\cdot$$\mu_{G}$,
with $f_{w,G}\in L^{1}(G)$. Since $\tau_{w,G}$ is conjugate invariant,
$f_{w,G}$ is a class function, and it can be written as a linear
combination of characters $f_{w,G}=\sum_{\rho\in\mathrm{Irr}(G)}\overline{a_{w,G,\rho}}\cdot\rho$. 

By Definition \ref{def:Concatination of words}, we see that $\tau_{w_{1}*w_{2},G}=\tau_{w_{1},G}*\tau_{w_{2},G}$
for every $w_{1}\in F_{r_{1}}$ and $w_{2}\in F_{r_{2}}$. Since $\rho_{1}*\rho_{2}=\frac{\delta_{\rho_{1},\rho_{2}}}{\rho_{1}(1)}\cdot\rho_{1}$
for every $\rho_{1},\rho_{2}\in\mathrm{Irr}(G)$, we have:
\begin{equation}
a_{w_{1}*w_{2},G,\rho}=\int_{G}\rho(g)\tau_{w_{1}*w_{2},G}(g)=\int_{G}\rho(g)\tau_{w_{1},G}*\tau_{w_{2},G}(g)=\frac{a_{w_{1},G,\rho}\cdot a_{w_{2},G,\rho}}{\rho(1)}.\label{eq:Fourier coefficients of convolutions of words}
\end{equation}

\begin{prop}
\label{prop:equidistribution- finite rank}For every $1\neq w\in F_{r}$
and $d\in\N$, there exists $\epsilon(d,w)>0$ such that:
\begin{enumerate}
\item For every compact connected semisimple Lie group $G$ of rank $d$
and every $\rho\in\mathrm{Irr}(G)$, we have $\left|a_{w,G,\rho}\right|\leq\rho(1)^{1-\epsilon(d,w)}$.
\item In particular, for every $1\leq m\leq d$,
\[
\mathbb{E}_{\U_{d}}\left(\left|\tr\left(\bigwedge\nolimits ^{m}w(X_{1},\ldots,X_{r})\right)\right|^{2}\right)=\mathbb{E}_{\mathrm{SU}_{d}}\left(\left|\tr\left(\bigwedge\nolimits ^{m}w(X_{1},\ldots,X_{r})\right)\right|^{2}\right)\leq\binom{d}{m}^{2(1-\epsilon(d,w))}.
\]
\end{enumerate}
\end{prop}

\begin{proof}
We first prove Item (1). Fix $w\in F_{r}$ and a compact connected
semisimple Lie group $G$. Let $\tau_{w,G}=f_{w,G}\mu_{G}$ be the
word measure. By (\ref{eq:Fourier coefficients of convolutions of words}),
and since $a_{w^{-1},G,\rho}=\overline{a_{w,G,\rho}}$ for each $\rho\in\mathrm{Irr}(G)$,
we have
\begin{equation}
a_{w*w^{-1},G,\rho}=\frac{\left|a_{w,G,\rho}\right|^{2}}{\rho(1)}.\label{eq:reduction to words with positive Fourier coefficients}
\end{equation}
Replacing $w$ by $w*w^{-1}$, we may assume that all Fourier coefficients
$a_{w,G,\rho}$ are in $\R_{\geq0}$.

It follows from \cite[Theorem 1.1]{GHS} that $f_{w,G}\in L^{1+\epsilon'}(G)$
for some $\epsilon'=\epsilon'(G,w)>0$. By Young's convolution inequality,
it follows that $f_{w,G}^{*t}\in L^{\infty}(G)$ for all $t\geq t_{0}(G,w):=\left\lceil \frac{1+\epsilon'(G,w)}{\epsilon'(G,w)}\right\rceil $
(see e.g.~\cite[Section 1.1, end of p.3]{GHS}). In particular, by
(\ref{eq:Fourier coefficients of convolutions of words}), we deduce
that: 
\[
f_{w,G}^{*t_{0}}(1)=\sum_{\rho\in\mathrm{Irr}(G)}\rho(1)^{2-t_{0}}a_{w,G,\rho}^{t_{0}}<\infty.
\]
Since $a_{w,G,\rho}\geq0$, we deduce that $a_{w,G,\rho}<\rho(1)^{1-\frac{2}{t_{0}(G,w)}}$
for all but finitely many $\rho\in\mathrm{Irr}(G)$. To deal with
the remaining finitely many (non-trivial) representations of $G$,
we simply use the bound $a_{w,G,\rho}<\rho(1)$, which follows e.g.~by
the It\^{o}\textendash Kawada equidistribution theorem \cite{IK40}
(see also \cite[Theorem 4.6.3]{App14}), since $\mathrm{Supp}(\tau_{w,G})$
generates $G$. Since there are only finitely many compact semisimple
connected Lie groups of rank $d$, this implies Item (1). 

Note that the character $\rho_{(1^{m})}\otimes\rho_{(1^{m})}^{\vee}$
of the representation $\bigwedge\nolimits ^{m}\C^{d}\otimes\left(\bigwedge\nolimits ^{m}\C^{d}\right)^{\vee}$
of $\mathrm{SU}_{d}$ is given by $\left|\tr\left(\bigwedge\nolimits ^{m}(A)\right)\right|^{2}$.
Since $\rho_{(1^{m})}\otimes\rho_{(1^{m})}^{\vee}$ is a sum of irreducible
characters, by applying the It\^{o}\textendash Kawada equidistribution
theorem to each irreducible character, for each $1\leq m\leq d$,
we have
\[
\mathbb{E}_{\mathrm{SU}_{d}}\left(\left|\tr\left(\bigwedge\nolimits ^{m}(w(X_{1},\ldots,X_{r}))\right)\right|^{2}\right)=\mathbb{E}_{\mathrm{SU}_{d}}\left(\rho_{(1^{m})}\otimes\rho_{(1^{m})}^{\vee}(w(X_{1},\ldots,X_{r}))\right)<\rho_{(1^{m})}\otimes\rho_{(1^{m})}^{\vee}(1)=\binom{d}{m}^{2}.
\]
Since there are only finitely many such $m$'s, this implies Item
(2).
\end{proof}
Theorem \ref{thm:main_thm_fundemental_rep} now follows from Proposition
\ref{prop:equidistribution- finite rank} and the following Theorem. 
\begin{thm}
\label{thm:bounds on absolute value of trace large d}For every $\ell\in\N$,
there exist $\epsilon(\ell),C(\ell)>0$ such that, for every $d\geq C(\ell)$,
every $1\leq m\leq d$, and every word $w\in F_{r}$ of length $\ell$,
one has: 
\[
\mathbb{E}\left(\left|\tr\left(\bigwedge\nolimits ^{m}w(X_{1},\ldots,X_{r})\right)\right|^{2}\right)\leq\binom{d}{m}^{2(1-\epsilon(\ell))}.
\]
\end{thm}

In order to prove Theorem \ref{thm:bounds on absolute value of trace large d},
we need the following technical lemma. 
\begin{lem}
\label{lem:Technical lemma}Let $H(x)=-x\log(x)-(1-x)\log(1-x)$ be
the binary entropy function. Then:
\begin{enumerate}
\item For every $d\in\N$ and every $0<x<1$ such that $dx\in\N$, we have
$\frac{2^{dH(x)}}{\sqrt{8dx(1-x)}}\leq\binom{d}{xd}\leq\frac{2^{dH(x)}}{\sqrt{\pi dx(1-x)}}\leq2^{dH(x)}$.
\item Let $0<\delta\leq\frac{1}{2}$. Then for every $b\in[\delta,\frac{1}{2}]$,
$a\in[\delta,b]$, and $d>\frac{1}{\delta^{4}}$ such that $bd,ad,d$
are integers, one has: 
\[
\binom{d}{(b-a)d}\leq\binom{d}{bd}^{1-\delta^{2}}.
\]
\end{enumerate}
\end{lem}

\begin{proof}
Item (1) follows e.g.~from \cite[Lemma 17.5.1]{CT06}. The Taylor
series of $H(x)$ around $1/2$ is
\begin{equation}
H(x)=1-\frac{1}{2\ln2}\sum_{n=1}^{\infty}\frac{(1-2x)^{2n}}{n(2n-1)}.\label{eq:Taylor series of entropy}
\end{equation}
Since $H'(x)=\log(\frac{1-x}{x})$, $H(x)$ is monotone increasing
in $(0,1/2)$, and therefore, 
\begin{align*}
H(b)-H(b-a) & \geq H(b)-H(b-\delta)=\frac{1}{2\ln2}\left(\sum_{n=1}^{\infty}\frac{(1-2b+2\delta)^{2n}-(1-2b)^{2n}}{n(2n-1)}\right)\\
 & \geq\frac{1}{2\ln2}\left((1-2b+2\delta)^{2}-(1-2b)^{2}\right)=\frac{1}{2\ln2}\left(4\delta^{2}+4\delta(1-2b)\right)\geq2\delta^{2}.
\end{align*}
Since $d>\frac{1}{\delta^{4}}\geq16$, we have $\frac{\log(d)}{d}\leq\frac{1}{\sqrt{d}}\leq\delta^{2}$.
Combining with Item (1), we have: 
\begin{align*}
\binom{d}{(b-a)d} & \leq2^{dH(b-a)}\leq\sqrt{8db(1-b)}2^{d\left(H(b-a)-H(b)\right)}\binom{d}{bd}\leq2^{-2d\delta^{2}+\log(d)}\binom{d}{bd}\\
 & \leq2^{-d\delta^{2}}\binom{d}{bd}=\left(2^{-dH(b)}\right)^{\frac{\delta^{2}}{H(b)}}\binom{d}{bd}\leq\binom{d}{bd}^{1-\frac{\delta^{2}}{H(b)}}\leq\binom{d}{bd}^{1-\delta^{2}}.\qedhere
\end{align*}
\end{proof}
\begin{proof}[Proof of Theorem \ref{thm:bounds on absolute value of trace large d}]
Since $\bigwedge\nolimits ^{m}V\simeq\left(\bigwedge\nolimits ^{d-m}V\right)^{\vee}\otimes\chi_{\mathrm{det}}$,
we may assume that $2m\leq d$. Let $\delta(\ell):=(25\ell)^{-\ell}$,
let $C(\ell)=\delta(\ell)^{-7}$, and suppose that $d\geq C(\ell)$.
By Theorem \ref{thm:abs_value_of_trace}, we may assume that $d\leq\delta(\ell)^{-1}m$,
and, in particular, $m\geq\delta(\ell)^{-6}$. As in the proof of
Proposition \ref{prop:equidistribution- finite rank}, by replacing
$w$ by $w*w^{-1}$, we may assume that $a_{w,\U_{d},\rho}\in\R_{\geq0}$
for all $\rho\in\mathrm{Irr}(\U_{d})$. By Theorem \ref{thm:Littlewood-Richardson-unitary},
we have for all $c\leq\frac{d}{2}$:
\[
\bigwedge\nolimits ^{c}V\otimes\bigwedge\nolimits ^{c}V^{\vee}\simeq\left(\bigwedge\nolimits ^{c}V\otimes\bigwedge\nolimits ^{d-c}V\right)\otimes\chi_{\mathrm{det}}^{-1}\simeq\bigoplus_{j=0}^{c}V_{\lambda_{(j)}},
\]
where $\lambda_{(j)}=(1,\ldots,1,0,\ldots,0,-1,\ldots,-1)$, with
$-1$ and $1$ appearing $j$ times. Moreover, $V_{\lambda_{(c)}}$
is the largest irreducible subrepresentation of $\bigwedge\nolimits ^{c}V\otimes\bigwedge\nolimits ^{c}V^{\vee}$,
and we have $\rho_{\lambda_{(c)}}(1)\geq\frac{1}{c+1}\binom{d}{c}^{2}\geq\binom{d}{c}^{3/2}$.
By Theorem \ref{thm:abs_value_of_trace}, and since all $a_{w,\U_{d},\rho}$
are non-negative, if $c\leq\lceil\delta(\ell)d\rceil\leq(22\ell)^{-\ell}d$,
then
\[
\mathbb{E}\left(\rho_{\lambda_{(c)}}\circ w\right)\leq\sum_{j=0}^{c}\mathbb{E}\left(\rho_{\lambda_{(j)}}\circ w\right)=\mathbb{E}\left(\rho_{\bigwedge\nolimits ^{c}V\otimes\bigwedge\nolimits ^{c}V^{\vee}}\circ w\right)=\mathbb{E}\left(\left|\left(\rho_{\bigwedge\nolimits ^{c}V}\circ w\right)\right|^{2}\right)\leq\binom{d}{c}\leq\rho_{\lambda_{(c)}}(1)^{2/3}.
\]
Applying the last inequality for $w^{*9}$, recalling that $a_{w^{*t},\U_{d},\rho}=\frac{a_{w,\U_{d},\rho}^{t}}{\rho(1)^{t-1}}$
for all $\rho\in\mathrm{Irr}(\U_{d})$, we get
\begin{equation}
\mathbb{E}\left(\left|\tr\left(\bigwedge\nolimits ^{\lceil\delta(\ell)d\rceil}w^{*9}(X_{1},\ldots,X_{r})\right)\right|^{2}\right)=\sum_{j=0}^{\lceil\delta(\ell)d\rceil}\mathbb{E}\left(\rho_{\lambda_{(j)}}\circ w^{*9}\right)\leq\sum_{j=0}^{\lceil\delta(\ell)d\rceil}\rho_{\lambda_{(j)}}(1)^{-2}\leq\sum_{j=1}^{\infty}\frac{1}{j^{2}}<2.\label{eq:bound.on.small.wedge.power}
\end{equation}
Note that, for each $\delta(\ell)d\leq m\leq\frac{d}{2}$, $\bigwedge\nolimits ^{m}V$
is a subrepresentation of $\bigwedge\nolimits ^{\lceil\delta(\ell)d\rceil}V\otimes\bigwedge\nolimits ^{m-\lceil\delta(\ell)d\rceil}V$,
so
\[
\bigwedge\nolimits ^{m}V\otimes\left(\bigwedge\nolimits ^{m}V\right)^{\vee}\hookrightarrow\left(\bigwedge\nolimits ^{\lceil\delta(\ell)d\rceil}V\otimes\left(\bigwedge\nolimits ^{\lceil\delta(\ell)d\rceil}V\right)^{\vee}\right)\otimes\left(\bigwedge\nolimits ^{m-\lceil\delta(\ell)d\rceil}V\otimes\left(\bigwedge\nolimits ^{m-\lceil\delta(\ell)d\rceil}V\right)^{\vee}\right).
\]
Finally, by the positivity of the Fourier coefficients of $w$, by
(\ref{eq:bound.on.small.wedge.power}), by Lemma \ref{lem:Technical lemma}
(note that $m\geq\lceil\delta(\ell)d\rceil$) and by (\ref{eq:Striling on binomial coefficients})
(note that $\delta(\ell)^{2}m\geq1$),
\begin{align*}
\mathbb{E}\left(\left|\tr\left(\bigwedge\nolimits ^{m}w^{*9}(X_{1},\ldots,X_{r})\right)\right|^{2}\right) & \leq\mathbb{E}\left(\left|\tr\left(\bigwedge\nolimits ^{\lceil\delta(\ell)d\rceil}w^{*9}(X_{1},\ldots,X_{r})\right)\right|^{2}\left|\tr\left(\bigwedge\nolimits ^{m-\lceil\delta(\ell)d\rceil}w^{*9}(X_{1},\ldots,X_{r})\right)\right|^{2}\right)\\
 & \leq\mathbb{E}\left(\left|\tr\left(\bigwedge\nolimits ^{\lceil\delta(\ell)d\rceil}w^{*9}(X_{1},\ldots,X_{r})\right)\right|^{2}\right)\cdot\binom{d}{m-\lceil\delta(\ell)d\rceil}^{2}\\
 & \leq2\binom{d}{m-\lceil\delta(\ell)d\rceil}^{2}\leq\frac{d}{m}\binom{d}{m}^{2-2\delta(\ell)^{2}}\leq\binom{d}{m}^{2-\delta(\ell)^{2}}.
\end{align*}
By (\ref{eq:Striling on binomial coefficients}), $m+1\leq2^{2\sqrt{m}}\leq\binom{d}{m}^{2/\sqrt{m}}$
for each $m\leq\frac{d}{2}$. Hence,
\begin{equation}
\rho_{\lambda_{(m)}}(1)\geq\frac{1}{m+1}\binom{d}{m}^{2}\geq\binom{d}{m}^{2(1-\frac{1}{\sqrt{m}})}\geq\binom{d}{m}^{2-2\delta(\ell)^{3}}.\label{eq:lower bound on dimension of irrep}
\end{equation}
Consequently, we get
\begin{align*}
\left(\mathbb{E}\left(\rho_{\lambda_{(m)}}\circ w\right)\right)^{9} & =\mathbb{E}\left(\rho_{\lambda_{(m)}}\circ w^{*9}\right)\rho_{\lambda_{(m)}}(1)^{8}\leq\mathbb{E}\left(\left|\tr\left(\bigwedge\nolimits ^{m}w^{*9}(X_{1},\ldots,X_{r})\right)\right|^{2}\right)\rho_{\lambda_{(m)}}(1)^{8}\\
 & \leq\binom{d}{m}^{2-\delta(\ell)^{2}}\rho_{\lambda_{(m)}}(1)^{8}\leq\rho_{\lambda_{(m)}}(1)^{9-\frac{\delta(\ell)^{2}}{4}},
\end{align*}
and thus $\mathbb{E}\left(\rho_{\lambda_{(m)}}\circ w\right)\leq\rho_{\lambda_{(m)}}(1)^{1-\frac{\delta(\ell)^{2}}{36}}$.
Taking $\epsilon(\ell):=\frac{\delta(\ell)^{2}}{72}$, and using $m+1\leq\binom{d}{m}^{2\delta(\ell)^{3}}$,
we get
\[
\mathbb{E}\left(\left|\tr\left(\bigwedge\nolimits ^{m}w(X_{1},\ldots,X_{r})\right)\right|^{2}\right)=\sum_{j=0}^{m}\mathbb{E}\left(\rho_{\lambda_{(j)}}\circ w\right)\leq(m+1)\binom{d}{m}^{2-4\epsilon(\ell)}\leq\binom{d}{m}^{2(1-\epsilon(\ell))}.\qedhere
\]
\end{proof}
We end the section with a proof of Theorem \ref{thm:LST theorem}. 
\begin{proof}[Proof of Theorem \ref{thm:LST theorem}]
Let $w\in F_{r}$. Denote $\widetilde{w}:=w*w^{-1}$. Recall that
$\tau_{\widetilde{w},G}=f_{\widetilde{w},G}\mu_{G}$ and note that
for every $t\in\N$, 
\[
f_{\widetilde{w}^{*t},G}=f_{\widetilde{w},G}^{*t}.
\]
Applying \cite[Theorem 4]{LST19}, there are $C',M(w)\in\N$ such
that, for $N(w):=C'\ell(w)^{4}$ and for every finite simple group
$G$ of size $>M(w)$, one has
\[
\left|\sum_{1\neq\rho\in\mathrm{Irr}(G)}\frac{a_{\widetilde{w},G,\rho}^{N(w)}}{\rho(1)^{N(w)-1}}\rho(1)\right|=\left|\sum_{1\neq\rho\in\mathrm{Irr}(G)}a_{\widetilde{w}^{*N(w)},G,\rho}\rho(1)\right|=\left|f_{\widetilde{w}^{*N(w)},G}(1)-1\right|=\left|f_{\widetilde{w},G}^{*N(w)}(1)-1\right|<1,
\]
where the first equality follows from (\ref{eq:Fourier coefficients of convolutions of words}).
Since $a_{\widetilde{w},G,\rho}=\frac{\left|a_{w,G,\rho}\right|^{2}}{\rho(1)}\geq0$,
we deduce that for each $1\neq\rho\in\mathrm{Irr}(G)$
\[
\frac{\left|a_{w,G,\rho}\right|^{2N(w)}}{\rho(1)^{2N(w)-2}}=\frac{\left|a_{w,G,\rho}\right|^{2N(w)}}{\rho(1)^{2N(w)-1}}\rho(1)=\frac{a_{\widetilde{w},G,\rho}^{N(w)}}{\rho(1)^{N(w)-1}}\rho(1)<1,
\]
from which the theorem follows for $\epsilon=\frac{1}{N(w)}=\frac{1}{C'\ell(w)^{4}}$. 
\end{proof}

\section{\label{sec:Fourier-coefficients-of symmetric powers}Fourier coefficients
of symmetric powers}

In this section, we prove Theorem \ref{thm:main thm symmetric powers}.
Denote $\mathcal{J}_{m,d}=\{c_{1}\leq\ldots\leq c_{m}:c_{i}\in[d]\}$.
We first claim that, for each $A\in\mathrm{End}(\C^{d})$ and $m\geq1$,
\[
\tr\left(\Sym^{m}A\right)=\frac{1}{m!}\sum_{\overrightarrow{a}\in[d]^{m}}\sum_{\pi\in S_{m}}A_{a_{1}a_{\pi(1)}}\cdots A_{a_{m}a_{\pi(m)}}.
\]
Indeed, for each $\overrightarrow{c}\in\mathcal{J}_{m,d}$, let $\nu_{\overrightarrow{c}}$
be the shape of $\overrightarrow{c}$ (see Definition \ref{def:shape of a function})
and set 
\[
v_{\overrightarrow{c}}:=\sqrt{\frac{1}{m!\cdot\nu_{\overrightarrow{c}}!}}\sum_{\pi\in S_{m}}e_{c_{\pi(1)}}\otimes\ldots\otimes e_{c_{\pi(m)}}.
\]
Then $\left\{ v_{\overrightarrow{c}}\right\} _{\overrightarrow{c}\in\mathcal{J}_{m,d}}$
is an orthonormal basis for $\Sym^{m}(\C^{d})$. Given $A\in\mathrm{End}(\C^{d})$,
we have: 
\begin{align*}
\tr\left(\Sym^{m}A\right) & =\sum_{\overrightarrow{c}\in\mathcal{J}_{m,d}}\langle A.v_{\overrightarrow{c}},v_{\overrightarrow{c}}\rangle=\sum_{\overrightarrow{c}\in\mathcal{J}_{m,d}}\frac{1}{m!\cdot\nu_{\overrightarrow{c}}!}\sum_{\pi,\pi'\in S_{m}}\langle Ae_{c_{\pi(1)}}\otimes\ldots\otimes Ae_{c_{\pi(m)}},e_{c_{\pi'(1)}}\otimes\ldots\otimes e_{c_{\pi'(m)}}\rangle\\
 & =\sum_{\overrightarrow{c}\in\mathcal{J}_{m,d}}\frac{1}{\nu_{\overrightarrow{c}}!}\sum_{\pi\in S_{m}}\langle Ae_{c_{1}}\otimes\ldots\otimes Ae_{c_{m}},e_{c_{\pi(1)}}\otimes\ldots\otimes e_{c_{\pi(m)}}\rangle\\
 & =\sum_{\overrightarrow{c}\in\mathcal{J}_{m,d}}\frac{1}{\nu_{\overrightarrow{c}}!}\sum_{\pi\in S_{m}}A_{c_{1}c_{\pi(1)}}\cdots A_{c_{m}c_{\pi(m)}}=\frac{1}{m!}\sum_{\overrightarrow{a}\in[d]^{m}}\sum_{\pi\in S_{m}}A_{a_{1}a_{\pi(1)}}\cdots A_{a_{m}a_{\pi(m)}},
\end{align*}
where the last equality follows since $\sum_{\pi\in S_{m}}A_{c_{1}c_{\pi(1)}}\cdots A_{c_{m}c_{\pi(m)}}$
is invariant under permuting $c_{1},\ldots,c_{m}$, and since there
are $\frac{m!}{\nu_{\overrightarrow{c}}!}$ vectors $\overrightarrow{a}\in[d]^{m}$
of a shape $\nu_{\overrightarrow{c}}$. In particular, for any word
$w$, 
\begin{equation}
\tr\left(\Sym^{m}w(X_{1},\ldots,X_{r})\right)=\frac{1}{m!}\sum_{\overrightarrow{a}\in[d]^{m}}\sum_{\pi\in S_{m}}w(X_{1},\ldots,X_{r})_{a_{1}a_{\pi(1)}}\cdots w(X_{1},\ldots,X_{r})_{a_{m}a_{\pi(m)}}.\label{eq:formula for trace of symmetric power}
\end{equation}

\begin{prop}
\label{prop:Main result-reduction to Weingarten- symmetric powers}Let
$w\in F_{r}$ be a cyclically reduced word. With $\Phi,T,\Omega,\Omega_{s,u}$
as in $\mathsection$\ref{sec:Reduction-of-Theorem}, we have:
\begin{equation}
\mathbb{E}\left(\left|\tr\left(\Sym^{m}w(X_{1},\ldots,X_{r})\right)\right|^{2}\right)=\frac{1}{m!^{2}}\sum_{(\pi,\pi',F,\varSigma)\in\widetilde{Z}}\widetilde{\Wg}(\varSigma^{2}),\label{eq:reduction to Weingarten symmetric powers}
\end{equation}
where 
\[
\widetilde{Z}:=\left\{ (\pi,\pi',F,\varSigma):\substack{F:\Omega\rightarrow[d],\varSigma\in S_{\Phi}\\
\pi,\pi'\in\Sym(\Omega_{1,1})\times\Sym(\Omega_{2,1})\\
F\circ T=F\circ\pi\pi'\circ\Sigma
}
\right\} .
\]
\end{prop}

\begin{proof}
Similarly to (\ref{eq:reformulation of trace}), we have 
\begin{align*}
\tr\left(\Sym^{m}w(X_{1},\ldots,X_{r})\right) & =\frac{1}{m!}\sum_{\overrightarrow{a}\in[d]^{m}}\sum_{\pi\in S_{m}}\sum_{\substack{f:[\ell+1]\times[m]\rightarrow[d]\\
f(1,k)=a_{k},f(\ell+1,k)=a_{\pi(k)}
}
}\prod_{(u,k)\in[\ell]\times[m]}\left(X_{w(u)}\right)_{f(u,k),f(u+1,k)}\\
 & =\sum_{\pi\in\Sym(\left\{ \ell\right\} \times[m])}\sum_{F:[\ell]\times[m]\rightarrow[d]}\prod_{(u,k)\in[\ell]\times[m]}\left(X_{w(u)}\right)_{F(u,k),F(\widetilde{T}\pi(u,k))}.
\end{align*}
Consequently, as in (\ref{eq:very simplified formula for trace}),
we have:
\[
\mathbb{E}\left(\left|\tr\left(\Sym^{m}w(X_{1},\ldots,X_{r})\right)\right|^{2}\right)=\frac{1}{m!^{2}}\sum_{(\pi,\pi')\in\Sym(\Omega_{1,1})\times\Sym(\Omega_{2,1})}\sum_{F:\Omega\rightarrow[d]}\prod_{\gamma\in\Omega}\left(X_{\widetilde{w}(\gamma)}\right)_{F(\pi\pi'\gamma),F(T(\gamma))}.
\]
The Proposition now follows from Corollary \ref{cor:Weingarten coordinate free- multidimensional}. 
\end{proof}
We next define an action of $H:=\prod_{(s,u)\in[2]\times[\ell]}\Sym(\Omega_{s,u})$
on $\widetilde{Z}$ in the same way as in $\mathsection$\ref{sec:Estimating-the-contribution}.
For $(s,u)\in[2]\times\left([\ell]\backslash\{1\}\right)$ and $\pi_{s,u}\in\Sym\left(\Omega_{s,u}\right)$,
\[
\pi_{s,u}.\left(\pi,\pi',F,\varSigma\right):=\left(\pi,\pi',F\circ\pi_{s,u}^{-1},\pi_{s,u}\circ\varSigma\circ T^{-1}\pi_{s,u}^{-1}T\right),
\]
and if $(\pi_{1,1},\pi_{2,1})\in\Sym(\Omega_{1,1})\times\Sym(\Omega_{2,1})$,
\[
\left(\pi_{1,1},\pi_{2,1}\right).\left(\pi,\pi',F,\varSigma\right):=\left(\pi_{1,1}\pi,\pi_{2,1}\pi',F\circ\pi_{1,1}^{-1}\pi_{2,1}^{-1},\varSigma\circ T^{-1}\pi_{1,1}^{-1}\pi_{2,1}^{-1}T\right).
\]

\begin{proof}[Proof of Theorem \ref{thm:main thm symmetric powers}]
The proof is similar to the proof of Theorem \ref{thm:abs_value_of_trace}.
The only difference is that now, summing over the $H$-orbit kills
all representations that do not appear in $\mathrm{Ind}_{S_{m}^{\ell_{i}}}^{S_{m\ell_{i}}}(1)$,
rather than the representations not in $\mathrm{Ind}_{S_{m}^{\ell_{i}}}^{S_{m\ell_{i}}}(\sgn)$.
By Lemma \ref{lem:description of littlewood richardson for Young subgroups},
the irreducible subrepresentations $\chi_{\lambda}$ of $\mathrm{Ind}_{S_{m}^{\ell_{i}}}^{S_{m\ell_{i}}}(1)$
correspond to partitions $\lambda=(\lambda_{1},\ldots,\lambda_{\ell_{i}})$
with at most $\ell_{i}$ rows, and, therefore, $\prod_{(a,b)\in\lambda}(d+b-a)\geq(d-\ell)^{m\ell_{i}}$.
As in Corollary \ref{cor:analyzing H-orbits} and (\ref{eq:Z.to.orbit}),
the average of $\widetilde{\Wg}(\varSigma^{2})$ over an $H$-orbit
$H.(\widehat{\pi},\widehat{\pi'},\widehat{F},\widehat{\varSigma})$
is bounded by
\begin{align}
 & \frac{1}{m!^{2\ell}}\prod_{i=1}^{r}\left|\sum_{\substack{h_{i}\in\prod_{\widetilde{w}=i}\Sym(\Omega_{s,u})\\
h'_{i}\in\prod_{\widetilde{w}=-i}\Sym(\Omega_{s,u})
}
}\Wg\left(h_{i}\widehat{\varSigma}|_{B_{i}}h'_{i}\widehat{\varSigma}|_{A_{i}}\right)\right|\nonumber \\
\leq & \frac{1}{m!^{\ell}}\prod_{i=1}^{r}\frac{m!^{\ell_{i}}}{(m\ell_{i})!}\sum_{\lambda\vdash m\ell_{i}:\chi_{\lambda}\subseteq\mathrm{Ind}_{\mathrm{S}_{m}^{\ell_{i}}}^{S_{m\ell_{i}}}(1)}\frac{\chi_{\lambda}(1)\langle\chi_{\lambda},1\rangle_{\mathrm{S}_{m}^{\ell_{i}}}}{\prod_{(a,b)\in\lambda}(d+b-a)}\leq\frac{1}{m!^{\ell}}\frac{1}{(d-\ell)^{m\ell}}.\label{eq:estimates on orbits- symmetric}
\end{align}
Denote $\widetilde{Z}_{\pi,\pi'}:=\left\{ (F,\varSigma):(\pi,\pi',F,\varSigma)\in\widetilde{Z}\right\} $.
Since $\widetilde{Z}_{\mathrm{Id},\mathrm{Id}}=W'$, Proposition \ref{prop:estimates on Z}
implies that
\begin{equation}
\left|\widetilde{Z}\right|=m!^{2}\left|\widetilde{Z}_{\mathrm{Id},\mathrm{Id}}\right|=m!^{2}\left|W'\right|\leq m!^{2}\binom{d+m\ell}{m\ell}(m\ell)!\prod_{0\neq k=-r}^{r}\frac{\left(m\ell_{k}\right)!}{\left(\sum_{i<k}m\ell_{i,k}\right)!}.\label{eq:estimates on Z symmetric}
\end{equation}
As in the proof of Theorem \ref{thm:abs_value_of_trace}, if $d\geq m\ell$,
then
\begin{align*}
\mathbb{E}\left(\left|\tr\left(\Sym^{m}w(X_{1},\ldots,X_{r})\right)\right|^{2}\right) & =\frac{1}{m!^{2}}\sum_{(\pi,\pi',F,\varSigma)\in\widetilde{Z}}\widetilde{\Wg}(\varSigma^{2})\leq\left|\widetilde{Z}\right|\frac{1}{m!^{\ell+2}}\frac{1}{(d-\ell)^{m\ell}}\\
 & \leq\frac{(d+m\ell)\cdots(d+1)}{(d-\ell)^{m\ell}m!^{\ell}}\prod_{0\neq k=-r}^{r}\frac{\left(m\ell_{k}\right)!}{\left(\sum_{i<k}m\ell_{i,k}\right)!}\\
 & \leq4^{m\ell}\ell^{m\ell}\prod_{0\neq k=-r}^{r}\binom{m\ell_{k}}{m\ell_{k}/2}4^{m\ell}\ell^{m\ell}2^{2m\ell}=(16\ell)^{m\ell}.\qedhere
\end{align*}
\end{proof}

\appendix

\section{\label{sec:Fourier-coefficients-of power word}Fourier coefficients
of the power word and a Diaconis\textendash Shahshahani type result}

In this Appendix, we formulate two results. The first is a computation
of the Fourier coefficients of the power word $w=x^{l}$ for representations
$\rho_{\lambda}\in\mathrm{Irr}\left(\U_{d}\right)$, where $\widetilde{\lambda}$
(see Remark \ref{rem:generalized Littlewood-Richardson}) has at most
$\frac{d}{2l}$ boxes. The second is a Diaconis\textendash Shahshahani
type result for the $m$-th coefficient of the characteristic polynomial
of a word $w$ in random unitary matrices. Both statements are consequences
of known results.
\begin{prop}
\label{prop:Fourier coeff of power}Let $w=x^{l}$ be the $l$-th
power word. Then, for every $m\in\N$ and every $d\geq2ml$, 
\begin{enumerate}
\item We have
\[
\mathbb{E}\left(\left|\rho_{\lambda}\circ w\right|^{2}\right)=\frac{1}{m!}\sum_{\sigma\in S_{m}}l^{\ell(\sigma)}\left|\chi_{\lambda}(\sigma)\right|^{2},
\]
for all $\lambda\vdash m$. In particular, $\mathbb{E}\left(\left|\rho_{\lambda}\circ w\right|^{2}\right)\leq l^{m}$.
\item We have
\[
\mathbb{E}\left(\left|\tr\left(\bigwedge\nolimits ^{m}w\right)\right|^{2}\right)=\mathbb{E}\left(\left|\tr\left(\Sym^{m}w\right)\right|^{2}\right)=\binom{l+m-1}{m}.
\]
\end{enumerate}
\end{prop}

\begin{proof}
For every matrix $A\in\U_{d}$ and every $\mu\vdash m$, set
\begin{equation}
\tr_{\mu}(A):=\prod_{j=1}^{m}\tr(A^{j})^{a_{j}},\label{eq:powersum}
\end{equation}
where $\mu=(1^{a_{1}}\cdots m^{a_{m}})$ is the partition $m=\underset{a_{1}\text{ times}}{\underbrace{(1+\ldots+1)}}+\ldots+\underset{a_{m}\text{ times}}{\underbrace{(m+\ldots+m)}}$.
The functions $\tr_{\mu}$ correspond to the power-sums symmetric
functions $p_{\mu}$. Given $\lambda\vdash m$, the character $\rho_{\lambda}(A)$
is a Schur polynomial in the eigenvalues of $A$, and, hence, it can
be expressed in terms of $\tr_{\mu}(A)$ via the following formula
(see e.g.~\cite[I.7, page 114]{Mac95}),
\begin{equation}
\rho_{\lambda}(A)=\sum_{\mu\vdash m}\frac{\chi_{\lambda}(\mu)}{\prod_{j=1}^{m}a_{j}!j^{a_{j}}}\cdot\tr_{\mu}(A),\label{eq:formula for characters}
\end{equation}
where $\chi_{\lambda}(\mu)$ is the value of the character $\chi_{\lambda}\in\mathrm{Irr}(S_{m})$
on the elements with cycle type $\mu$. In addition, by (\ref{eq:DS moments of trace}),
for every pair of partitions $\mu=(1^{a_{1}}\cdots m^{a_{m}})$ and
$\mu'=(1^{b_{1}}\cdots m^{b_{m}})$ of $m$, we have:
\begin{equation}
\mathbb{E}\left(\tr_{\mu}(X^{l})\tr_{\mu'}(\overline{X}^{l})\right)=\mathbb{E}\left(\prod_{j=1}^{m}\tr(X^{jl})^{a_{j}}\tr(\overline{X}^{jl})^{b_{j}}\right)=\delta_{\mu,\mu'}\prod_{j=1}^{m}(jl)^{a_{j}}a_{j}!.\label{eq:orthogonality of power-sums}
\end{equation}
Combining (\ref{eq:formula for characters}) and (\ref{eq:orthogonality of power-sums}),
and using the fact that the number of permutations $\sigma\in S_{m}$
of cycle type $\mu=(1^{a_{1}}\cdots m^{a_{m}})$ is $\frac{m!}{\prod_{j=1}^{m}a_{j}!j^{a_{j}}}$,
we obtain:
\begin{align}
\mathbb{E}\left(\left|\rho_{\lambda}\left(X^{l}\right)\right|^{2}\right) & =\sum_{\mu\vdash m}\left|\chi_{\lambda}(\mu)\right|^{2}\frac{\mathbb{E}\left(\left|\tr_{\mu}\left(X^{l}\right)\right|^{2}\right)}{\left(\prod_{j=1}^{m}a_{j}!j^{a_{j}}\right)^{2}}=\sum_{\mu\vdash m}\frac{l^{\ell(\mu)}\left|\chi_{\lambda}(\mu)\right|^{2}}{\prod_{j=1}^{m}a_{j}!j^{a_{j}}}\nonumber \\
 & =\frac{1}{m!}\sum_{\mu\vdash m}\frac{m!}{\prod_{j=1}^{m}a_{j}!j^{a_{j}}}l^{\ell(\mu)}\left|\chi_{\lambda}(\mu)\right|^{2}=\frac{1}{m!}\sum_{\sigma\in S_{m}}l^{\ell(\sigma)}\left|\chi_{\lambda}(\sigma)\right|^{2}.\label{eq:Fourier coeffcients of the power word}
\end{align}
The second claim of Item (1) follows from Schur orthogonality and
the inequality $l^{\ell(\sigma)}\leq l^{m}$. 

For Item (2), note that $\tr\left(\bigwedge\nolimits ^{m}w\right)=\rho_{(1^{m})}\circ w$
and $\tr\left(\Sym^{m}w\right)=\rho_{(m^{1})}\circ w$. The corresponding
characters of $S_{m}$ are the sign and the trivial characters. Thus,
(\ref{eq:Fourier coeffcients of the power word}) becomes
\[
\mathbb{E}\left(\left|\tr\left(\bigwedge\nolimits ^{m}w\right)\right|^{2}\right)=\mathbb{E}\left(\left|\tr\left(\Sym^{m}w\right)\right|^{2}\right)=\mathbb{E}_{S_{m}}\left(l^{\ell(\sigma)}\right)=\frac{1}{m!}\sum_{k=1}^{m}\left[\begin{array}{c}
m\\
k
\end{array}\right]l^{k}=\binom{l+m-1}{m},
\]
where $\left[\begin{array}{c}
m\\
k
\end{array}\right]$ is the number of permutations of $m$ elements with exactly $k$
disjoint cycles, also known as the unsigned Stirling number of the
first kind. The last equality follows e.g.~from \cite[Equation (6.11)]{GKP94}.
This concludes Item (2). 
\end{proof}
We next prove a Diaconis\textendash Shahshahani type result. We first
recall the following proposition, which is a consequence of \cite[Theorem 2]{MSS07}
and \cite[Theorem 4.1]{Rud06} (see also \cite[Corollary 1.13]{MP19}). 
\begin{prop}
\label{prop:power sums for general words}Let $w\in F_{r}$, and let
$\mu=(1^{a_{1}}\cdots m^{a_{m}})$, $\mu'=(1^{b_{1}}\cdots m^{b_{m}})$
be partitions of $m$. Let $p(w)\in\N$ be such that $w=u^{p(w)}$
with $u\in F_{r}$ a non-power. Then:
\begin{equation}
\underset{d\rightarrow\infty}{\lim}\mathbb{E}_{\U_{d}}\left(\tr_{\mu}(w)\tr_{\mu'}(w^{-1})\right)=\underset{d\rightarrow\infty}{\lim}\mathbb{E}_{\U_{d}}\left(\prod_{j=1}^{m}\tr(w^{j})^{a_{j}}\tr(w^{-j})^{b_{j}}\right)=\delta_{\mu,\mu'}\prod_{j=1}^{m}a_{j}!(jp(w))^{a_{j}}.\label{eq:moments of product of words}
\end{equation}
\end{prop}

Since the joint moments of $\tr(w^{1}),\ldots,\tr(w^{m})$ converge,
as $d\rightarrow\infty$, to the joint moments of independent complex
normal random variables, an application of the moment method (as was
done in \cite{DS94} for $w=x$, and later in \cite{Rud06,MSS07}
for a general word) implies
\begin{cor}[{\cite[Theorem 4.1]{Rud06}, \cite[Theorem 2]{MSS07}}]
\label{cor:The-random-variables}The random variables $\tr(w^{1}),\ldots,\tr(w^{m})$
converge in distribution to $\sqrt{p(w)}Z_{1},\ldots,\sqrt{mp(w)}Z_{m}$,
as $d\rightarrow\infty$, where $Z_{1},\ldots,Z_{m}$ are independent
complex normal variables.
\end{cor}

In \cite{DG04}, Diaconis and Gamburd combined Corollary \ref{cor:The-random-variables}
for $w=x$ (namely \cite{DS94}), together with Newton's identities
relating elementary and power sum symmetric functions to give a formula
for the limit behavior of the random variables $\tr\bigwedge\nolimits ^{m}X$
with $X$ is a random unitary matrix in $\U_{d}$. Repeating the argument
for a general word $w$ yields the following description of $\underset{d\rightarrow\infty}{\lim}\tr_{\U_{d}}\bigwedge\nolimits ^{m}w$. 
\begin{cor}[{cf.~\cite[Proposition 4]{DG04}}]
\label{cor:moment method}Let $w\in F_{r}$ be a word and let $m\in\N$.
Then the sequence of random variables $\tr_{\U_{d}}\bigwedge\nolimits ^{m}w$
converges in distribution, as $d\rightarrow\infty$, to the polynomial
in the normal variables $Z_{1},\ldots,Z_{m}$ given by: 
\[
\frac{1}{m!}\det\left(\begin{array}{ccccc}
\sqrt{p(w)}Z_{1} & 1 & 0 & \ldots & 0\\
\sqrt{2p(w)}Z_{2} & \sqrt{p(w)}Z_{1} & 2 & \ldots & 0\\
\vdots & \vdots & \vdots & \ddots & \vdots\\
\sqrt{(m-1)p(w)}Z_{m-1} & \sqrt{(m-2)p(w)}Z_{m-2} & \sqrt{(m-3)p(w)}Z_{m-3} & \ldots & (m-1)\\
\sqrt{mp(w)}Z_{m} & \sqrt{(m-1)p(w)}Z_{m-1} & \sqrt{(m-2)p(w)}Z_{m-2} & \ldots & \sqrt{p(w)}Z_{1}
\end{array}\right).
\]
\end{cor}

\begin{example}
Let $m=3$. Then for every Borel set $A\subseteq\C$, 
\[
\underset{d\rightarrow\infty}{\lim}\mathbb{P}\left(\tr_{\U_{d}}\bigwedge\nolimits ^{3}w(X_{1},...,X_{r})\in A\right)=\mathbb{P}\left(f(Z_{1},Z_{2},Z_{3})\in A\right),
\]
where $Z_{1},Z_{2},Z_{3}$ are i.i.d normal variables, and
\[
f(Z_{1},Z_{2},Z_{3})=\frac{p(w)^{3/2}}{6}Z_{1}^{3}-\frac{p(w)}{\sqrt{2}}Z_{1}Z_{2}+\frac{p(w)^{1/2}}{\sqrt{3}}Z_{3}.
\]
\end{example}

\bibliographystyle{alpha}
\bibliography{bibfile}

\end{document}